\documentclass[11pt]{article}

\usepackage[margin=1.08in]{geometry}
\usepackage{amsmath,amssymb,amsthm}
\usepackage{mathtools}
\usepackage{booktabs}
\usepackage{array}
\usepackage{enumitem}
\usepackage{xcolor}
\usepackage[colorlinks=true,linkcolor=blue!50!black,
  citecolor=green!35!black,urlcolor=blue!55!black]{hyperref}
\usepackage{microtype}
\usepackage{listings}

\allowdisplaybreaks
\setlist{itemsep=2pt,topsep=4pt}
\newtheorem{theorem}{Theorem}[section]
\newtheorem{lemma}[theorem]{Lemma}
\newtheorem{proposition}[theorem]{Proposition}
\newtheorem{corollary}[theorem]{Corollary}

\theoremstyle{definition}

\theoremstyle{remark}
\newtheorem{remark}[theorem]{Remark}

\newcommand{\F}{\mathbb{F}}
\newcommand{\R}{\mathbb{R}}
\newcommand{\E}{\mathbb{E}}
\newcommand{\body}{\mathcal{B}}
\newcommand{\one}{\mathbf{1}}

\newcommand{\down}{\downarrow}

\title{Value distributions for read-once polynomials on finite fields}
\author{A.~Yashunsky and
D.~Tabalin\textsuperscript{\hyperlink{author-contributions}{1}}}
\date{July 2026}
\hypersetup{
  pdftitle={Value distributions for read-once polynomials on finite fields},
  pdfauthor={A. Yashunsky and D. Tabalin},
  pdfsubject={A body of probability distributions stable under additive and
multiplicative convolution over finite fields}
}

\begin{document}
\maketitle
\footnotetext[1]{\hypertarget{author-contributions}{}Author contributions
and AI disclosure: A.~Yashunsky pointed out the problem to D.~Tabalin and
provided expert review; D.~Tabalin was the model whisperer. The authors made substantial use of generative AI
tools: GPT 5.6 Sol found the proof and produced the Lean formalization and
the initial exposition. Claude 5 Fable reviewed the results and improved the exposition.
The human authors take full responsibility for the statements,
proofs, and presentation.}
\setcounter{footnote}{1}

\begin{abstract}
Consider read-once polynomial functions over a finite field of
order~$k$, i.e., functions defined by expressions built from field
addition, multiplication, and constants, in which every variable
occurs at most once. Let $p=(p_1,\dots,p_k)$ be the distribution of
the values of a read-once function on independent uniform inputs.
We prove that for every $k\ge4$ all such distributions belong to a
body $\body_k$, defined by the following relation on the sorted
atoms $p_1^\down\ge\cdots\ge p_k^\down$ of the distribution:
\[
  \body_k=\left\{p:p_k^\down\ge
  \frac{1-p_2^\down-(1-p_2^\down)^k}{k-1}\right\},
\]
or equivalently,
$\body_k = \{ p \colon 1-p_2^\down-(k-1)p_k^\down\le(1-p_2^\down)^k\}$.

Our main theorem is that this body is stable under the convolutions
corresponding to both field operations.  More generally, convolution
for any quasigroup operation on $k$ points preserves $\body_k$; the
multiplicative conclusion needs only an absorbing zero and a
quasigroup operation on the nonzero elements.

The body is full-dimensional and contains the uniform law and every
point mass; its normalized volume is given by an exact one-dimensional
integral, and the $k$th root of that volume tends to
$0.2183305369\ldots$.  The complete development, including
the two stability theorems, the exact volume formula, and its sharp exponential rate,
has been formalized in Lean 4.
\end{abstract}

\section{Introduction}\label{sec:introduction}

The initial motivation for this work was a simple question: how do
the laws of independent random elements of a finite field transform
under the field operations?  If $X\sim\mu$ and $Y\sim\nu$ are
independent, the laws of $X+Y$ and $XY$ are the additive and
multiplicative convolutions $\mu\oplus\nu$ and $\mu\otimes\nu$, and a
basic program, going back to work of Yashunsky on operation-stable
polyhedra of distributions~\cite{yashunsky2015, yashunskii2017}, asks for regions of
the probability simplex that both convolutions preserve.  Such stable
regions are the natural invariants of arithmetic computation with
independent inputs: whatever a computation built from $+$ and $\times$
does, its output law cannot leave the region.

The independence required at every operation arises naturally in a
\emph{read-once} arithmetic formula: a binary $\{+,\times\}$ tree, with
variables and constants at the leaves, in which no input variable occurs
twice.  Distinct subtrees then involve
disjoint variables, hence are independent when the inputs are sampled
independently, and the law at every gate is the additive or
multiplicative convolution of the laws below it.  Any
convolution-stable region containing the laws of the inputs therefore
absorbs the output law of every read-once formula.  Statements about
stable bodies thus become statistical statements about read-once
computations, provided that the uniform distribution --- the law of a
single variable --- belongs to the body. Traditionally, the central
question on read-once functions is whether a specific function, either Boolean
or arithmetic, has a read-once representation (see \cite{golumbic2011,volkovich2016},
for instance). Yet we believe our approach is of value as well,
contributing to a general understanding of how read-once functions work.

This paper constructs such a body for every order $k\ge4$.  It contains
all point masses and the uniform law (Proposition~\ref{prop:B-base-laws}),
so its stability bounds the value distributions of read-once polynomials
on independent uniform inputs, constants allowed
(Corollary~\ref{cor:readonce}).  For $k\ge2$ let
\begin{equation}\label{eq:psi-def}
  \psi_k(t):=\frac{1-t-(1-t)^k}{k-1},
\end{equation}
sort the atoms of a law $p$ on $k$ points as
$p_1^\down\ge\cdots\ge p_k^\down$, and define
\begin{equation}\label{eq:Bq-def}
  p\in\body_k\quad\Longleftrightarrow\quad
  p_k^\down\ge\psi_k(p_2^\down).
\end{equation}
Our main result is the following stability theorem.

\begin{theorem}[stability of the body]\label{thm:body}
Let $k\ge4$.  Additive convolution preserves $\body_k$ on every finite
abelian group of order $k$.  Multiplicative convolution preserves
$\body_k$ on every finite field of order $k$.
\end{theorem}

The headline claim of the abstract is its immediate corollary.

\begin{corollary}[read-once value distributions]\label{cor:readonce}
Let $\F$ be a finite field of order $k\ge4$.  The value distribution of
every read-once polynomial function over $\F$ on independent uniform
inputs belongs to $\body_k$.
\end{corollary}

\begin{proof}
Induct on the formula tree.  The law of a leaf is the uniform law (a
variable) or a point mass (a constant), and both lie in $\body_k$ by
Proposition~\ref{prop:B-base-laws}.  At every gate the two subtrees
involve disjoint variables, so their laws are independent and the output
law is their additive or multiplicative convolution, which remains in
$\body_k$ by Theorem~\ref{thm:body}.
\end{proof}

Call a binary operation $\star$ on a finite set $Q$ a
\emph{quasigroup operation} if every left and right translation
\[
 x\longmapsto a\star x,\qquad x\longmapsto x\star a
\]
is a permutation of $Q$.  If $Q$ has a distinguished element $0$, call
an operation $\diamond$ an \emph{absorbing-zero quasigroup operation} if
$0\diamond x=x\diamond0=0$ and the restriction of $\diamond$ to
$Q^\times:=Q\setminus\{0\}$ is a quasigroup operation.  Equivalently, under
absorption, left and right translation by every nonzero element is injective
on $Q$.

\begin{theorem}[quasigroup extension]\label{thm:quasigroup}
Let $Q$ have order $k\ge4$.
\begin{enumerate}
\item Convolution with respect to any quasigroup operation on $Q$
preserves $\body_k$.
\item Convolution with respect to any absorbing-zero quasigroup operation
on $Q$ preserves $\body_k$.
\end{enumerate}
Consequently, an ordinary quasigroup addition and an absorbing-zero
quasigroup multiplication both preserve $\body_k$, with no compatibility
law between the two operations.
\end{theorem}

Theorem~\ref{thm:body} extends to every order $k\ge4$ the main theorem
of~\cite{yashunsky2021}, which settles the 3-element case of this
program: the value distribution of a read-once polynomial over a
3-element field --- the law of its output on independent uniform
inputs --- satisfies
$\max p_i - \min p_i \le (\max p_i + \min p_i)^3$, and the proof
exhibits a region, stable under both operations, whose membership
condition is exactly this inequality.  The body recovers that region at
$k=3$: unfolding \eqref{eq:psi-def} in \eqref{eq:Bq-def} writes
membership in $\body_k$ as
$1-p_2^\down-(k-1)p_k^\down\le(1-p_2^\down)^k$, and for $k=3$,
eliminating the middle atom, this becomes
$p_1^\down - p_3^\down \le (p_1^\down + p_3^\down)^3$.  The extension
was never intended to be the naive one: the region defined by
$p_1^\down-p_k^\down\le(1-p_2^\down)^k$ itself is not preserved at
larger orders --- already over $\F_5$ it fails to be closed under
multiplicative convolution (Remark~\ref{rem:obstruction}) --- so the
task was to find a definition of the body that works for every
$k\ge4$ and matches the known region at $k=3$.

Although the proofs in~\cite{yashunsky2021} take clear advantage of the
very specific structure of a 3-element field, Theorem~\ref{thm:quasigroup}
makes precise that the present criterion is indifferent to the finer
algebraic properties of the carrier.  The proof uses only the permutation
structure of the left and right translations, together with absorption at
zero in the multiplicative case.  In particular, the group laws that
motivate the field formulation of Theorem~\ref{thm:body} are not needed for
the stability argument itself.

The AI-generated proof is clearly distinct from the one
in~\cite{yashunsky2021}, and the body itself is not an obvious
generalization of the 3-element case.  Yet some traits seem to be
inherited.  The initial, non-AI attempts at a generalization were
rooted in the following observation.  On the part of the simplex where
the zero element of the field carries the middle probability, so that
$p_2^\down$ is its atom, the proved inequality reads exactly
$p_1^\down - p_3^\down \le (1-p_2^\down)^3$, which may be rewritten as
\begin{equation}\label{eq:3elem-ineq}
\frac{p_1^\down}{1-p_2^\down} - \frac{p_3^\down}{1-p_2^\down}
\le (1-p_2^\down)^2 .
\end{equation}
The left-hand side is the spread of the conditional probabilities on
the multiplicative group of the field --- a measure of non-uniformity
of the conditional law, approaching $0$ as that law approaches the
uniform one.  The closure of the 3-element body under field
multiplication of random variables on this part of the simplex is due
to the fact that the spread is submultiplicative with respect to
convolution of distributions on the multiplicative group.  Hence the
attempts at generalizing the 3-element result to bigger fields focused
on finding a non-uniformity measure for the multiplicative group that
would be submultiplicative with respect to convolution, because the
spread itself no longer is once $k\ge5$.  The
paper~\cite{yashunsky2024} finds such a measure for arbitrary~$k$, but
further progress stalled there.

To our surprise, a much simpler non-uniformity measure, introduced
in~\cite{yashunsky2023}, surfaces in the definition of $\body_k$ and
lets one trace the parallels between the 3-element and the $k$-element
case.  Write $n:=k-1$; over a field $\F$ of order $k$ this is the order
of the multiplicative group $\F^\times$.  For a distribution
$x=(x_1,\dots,x_n)$ on $n$ points, define $d_0(x) = 1 - n\min_i x_i$.
By~\cite[Theorem~5]{yashunsky2023}, if $*$ is the convolution of
distributions with respect to a quasigroup operation --- precisely the
generality used in Theorem~\ref{thm:quasigroup} --- then
\[
 d_0(x * y)\le d_0(x)\,d_0(y).
\]

Through this inequality, a zero-anchored companion of the body is
closed under multiplicative convolution outright.  For a law $x$
on~$\F$, write $x(0)$ for its zero atom and, when $x(0)<1$, let $x'$ be
the conditional law on $\F^\times$; consider
\[
 \mathcal Z_k:=\Bigl\{x:\ \min_{g\ne0}x(g)\ \ge\ \psi_k\bigl(x(0)\bigr)\Bigr\},
\]
which for $x(0)<1$ is the condition
$d_0(x')\le(1-x(0))^{n}$.  If $x$ and $y$ are the laws of independent
random variables, their product is zero with probability
$z=1-(1-x(0))(1-y(0))$, and conditionally on being nonzero its law is
$x'*y'$, the group convolution on $\F^\times$.  Submultiplicativity
then closes the region in one line:
\[
 d_0(x'*y')\le d_0(x')\,d_0(y')
 \le(1-x(0))^{n}(1-y(0))^{n}=(1-z)^{n},
\]
so multiplicative convolution preserves $\mathcal Z_k$, with no
assumption on where the zero atom ranks among the others; the same
mechanism drives the multiplication-preserved families
of~\cite{yashunsky2021}.  Dividing the membership condition
$1-p_2^\down-(k-1)p_k^\down\le(1-p_2^\down)^k$ by $1-p_2^\down$ shows
that $\body_k$ coincides with $\mathcal Z_k$ on the part of the simplex
where the zero atom is the second largest.  Additive convolution,
however, rules $\mathcal Z_k$ out as the invariant: convolving with a
point mass permutes the atoms and moves the zero anchor (the uniform
law on $\F^\times$ lies in $\mathcal Z_k$; its translate by a nonzero
element does not), so a region stable under both operations cannot be
anchored at the zero element.  The body is exactly the sorted
transplant of the zero-anchored shape: anchored at the second and
smallest atoms, it agrees with $\mathcal Z_k$ where the zero atom is
second largest and extends the same shape to the rest of the simplex
--- closely matching how the preserved region of the 3-element case is
assembled from a zero-anchored family by the symmetries of the simplex.
Even the proof difficulties appear in the same places: under
multiplication the zero atom can only grow, so products drift toward
the part of the simplex where the zero atom is largest; there the
zero-anchored bound genuinely fails, and that is precisely the K-type
regime of the dichotomy of Section~\ref{sec:mul-reduction}.

The body is a full-dimensional region of the simplex containing the
uniform law and all point masses (Propositions~\ref{prop:B-base-laws}
and~\ref{prop:B-interior}), and its boundary carries the one-parameter
sharp family $D_k(u)$ of Section~\ref{subsec:sharp}, on which the
multiplicative estimates are exact.  Writing $\mathfrak v_k$ for its
normalized relative volume, one has an explicit one-dimensional integral
for $\mathfrak v_k$ and the sharp exponential rate
$\mathfrak v_k^{1/k}\to0.2183305369\ldots$
(Appendix~\ref{app:exact-volume}).

The proof of Theorem~\ref{thm:body} via Theorem~\ref{thm:quasigroup} has three parts.

\begin{enumerate}
\item Section~\ref{sec:invariant} develops the scalar function
$\psi_k$, the cap $\bar t_k$, the quantitative nonvacuity of $\body_k$,
the distinguished base laws, and the sharp family.
\item Section~\ref{sec:addition} proves the ordinary quasigroup gate by
rearrangement, a tangent expansion, a crossing argument, and one
polynomial-band inequality.
\item Sections~\ref{sec:mul-reduction}--\ref{sec:assembly} prove the
absorbing-zero quasigroup gate.  The zero atom is separated from the
nonzero quasigroup, each factor is classified as K or D according to whether
its zero atom is largest, and rearrangement reduces the gate to K$\times$K,
K$\times$D, and D$\times$D scalar cores.
\end{enumerate}

The exact relative-volume formula and its sharp exponential scale are derived
separately in Appendix~\ref{app:exact-volume}.
Appendix~\ref{sec:lean} describes the complete formalization of these results
in Lean~4.

\section{The body \texorpdfstring{$\body_k$}{Bk}}\label{sec:invariant}

This section collects the notation, the scalar facts about $\psi_k$,
and the basic geometry of the body: its distinguished laws, its quantitative
nonvacuity, and the sharp boundary family on which the later estimates are
tight.  The exact volume derivation is deferred to
Appendix~\ref{app:exact-volume}.

\subsection{Notation}\label{subsec:notation}

Throughout the paper, $k$ is the order of the underlying structure and
$n:=k-1$.  We write $Q$ for a finite carrier, $\star$ for an ordinary
quasigroup operation, and $\diamond$ for an absorbing-zero quasigroup
operation.  For the field specialization, $\F$ is a finite field of order
$k$, $\F^\times$ its multiplicative group, and $G$ a finite abelian group
of order $k$.  We write
$\Delta_{k-1}$ for the probability simplex on $k$ points.  The letters
$\mathrm K$ and $\mathrm D$ are reserved for the two input \emph{types} of
the multiplication argument (Section~\ref{sec:mul-reduction}); they are
labels, never variables.  Probability laws are $\mu,\nu$, and
$p_i^\down(\mu)$ denotes the $i$th largest atom of $\mu$, with the
argument omitted when the law is clear from context.  When sorted atoms
need to be distinguished from coordinates in their original labeling, we
use the superscript $\down$, as in $p_1^\down\ge\cdots\ge p_k^\down$; for
a vector explicitly declared sorted, we may omit it.

For laws $\mu,\nu$ on a finite abelian group (respectively, on a finite
field), write $\mu\oplus\nu$ and $\mu\otimes\nu$ for the laws of $X+Y$
and $XY$ for independent $X\sim\mu$, $Y\sim\nu$:
\[
 (\mu\oplus\nu)(g)=\sum_{h+h'=g}\mu(h)\nu(h'),\qquad
 (\mu\otimes\nu)(g)=\sum_{hh'=g}\mu(h)\nu(h').
\]
For a general operation, its convolution is defined by the same fiber sum.
In Sections~\ref{sec:addition}--\ref{sec:assembly}, the operation is clear
from context; we retain $\otimes$ for the absorbing-zero operation
$\diamond$.

Table~\ref{tab:notation} collects every symbol that is shared between
sections; each is also introduced in running text where it first appears.

\begin{table}[p]
\centering\small
\begin{tabular}{@{}l@{\quad}p{7.2cm}@{\quad}l@{}}
\toprule
Symbol & Meaning & Introduced\\
\midrule
\multicolumn{3}{@{}l}{\emph{Global}}\\[1pt]
$k$;\ $n=k-1$ & order of the structure ($k\ge4$); order of $\F^\times$ &
\S\ref{subsec:notation}\\
$\F$, $\F^\times$;\ $G$ & field of order $k$, its multiplicative group;
abelian group of order $k$ & \S\ref{subsec:notation}\\
$Q$;\ $\star$, $\diamond$ & finite carrier; ordinary quasigroup and
absorbing-zero quasigroup operations & \S\ref{subsec:notation}\\
$\mu,\nu$;\ $\oplus$, $\otimes$ & the two input laws; additive and
multiplicative convolution & \S\ref{subsec:notation}\\
$p_i^\down(\mu)$ & $i$th largest atom of $\mu$ & \S\ref{subsec:notation}\\
$\Delta_{k-1}$ & probability simplex on $k$ points &
\S\ref{subsec:notation}\\
$\psi_k$;\ $\delta_k$ & the scalar function \eqref{eq:psi-def}; its gap
$\delta_k(t)=t-\psi_k(t)$ & \S\ref{subsec:scalar}\\
$t_k^*$;\ $\bar t_k$ & end of the increasing branch \eqref{eq:tstar};
the cap \eqref{eq:cap-root} & \S\ref{subsec:scalar}\\
$\body_k$ & the body & \eqref{eq:Bq-def}\\
$D_k(u)$;\ $\mathcal H_k$ & sharp boundary family; its complementary
atom & \eqref{eq:sharp-family}, \eqref{eq:sharp-H}\\
$\mathfrak v_k$;\ $c_*$, $\Lambda_*$, $\beta_*$ & relative volume of
$\body_k$; its saddle point, rate, and root limit &
\S\ref{sec:invariant}, App.~\ref{app:exact-volume}\\
$\mathrm K$, $\mathrm D$ & input types: zero atom maximal / not &
\S\ref{sec:mul-reduction}\\
$A$, $B$, $C$ & anti-aligned, runner-up, and aligned pairings: of full
$k$-vectors in Section~\ref{sec:addition}, of nonzero blocks in
Sections~\ref{sec:mul-reduction}--\ref{sec:assembly} &
\eqref{eq:add-AB}, \eqref{eq:mul-AC}--\eqref{eq:mul-B}\\
\midrule
\multicolumn{3}{@{}l}{\emph{Addition gate (Section~\ref{sec:addition});
$x,y$ are the sorted input $k$-vectors, $w$ a generic sorted member}}\\[1pt]
$T_w$;\ $C_i(w)$, $\Sigma C(w)$ & tangent margin of the bottom pair;
tangent coefficients and their total &
\eqref{eq:add-K}--\eqref{eq:add-sumC}\\
$U_1$, $U_2$;\ $F(w)$ & the two tangent lower bounds; the single-member
margin & Lemma~\ref{lem:add-tangent}, \eqref{eq:add-single}\\
\midrule
\multicolumn{3}{@{}l}{\emph{Multiplication gate
(Sections~\ref{sec:mul-reduction}--\ref{sec:assembly}); symbols for
$\mu$, analogues for $\nu$ in parentheses}}\\[1pt]
$a$ ($b$);\ $r$ ($s$) & zero atom; nonzero mass $1-a$ &
\eqref{eq:floor-residual-one}\\
$m$ ($\ell$) & global floor $p_k^\down(\mu)$ &
\eqref{eq:floor-residual-one}\\
$x$ ($y$);\ $\hat x$ ($\hat y$);\ $R$ ($S$) & sorted nonzero block;
residual $x-m\one$; residual total & \eqref{eq:floor-residual-one}\\
$z$;\ $c$;\ $w$ & product zero atom $1-rs$; product baseline; product
nonzero vector $c\one+\hat x*\hat y$ & \eqref{eq:mul-baseline},
\eqref{eq:nonzero-convolution}\\
$p$ ($v$);\ $U$ ($V$) & largest nonzero atom; top residual $p-m$ &
\S\S\ref{sec:mul-reduction}--\ref{sec:K-core}\\
$E$;\ $\Theta$;\ $\lambda$;\ $\kappa$ & alignment excess $\min(RV,SU)$;
aligned bound $c+E$; effective rank; reduced rank &
\eqref{eq:KK-floor}--\eqref{eq:KK-T}\\
$h$;\ $\varphi$ & normalized gaps $\delta_k(t)/t^2$ and $\psi_k(t)/t$ &
\eqref{eq:h-def}, Lemma~\ref{lem:K-rank}\\
$s_0$;\ $\eta$ & sharp mixed factor $S^{1/k}$; its floor $\psi_k(1-s_0)$
& \eqref{eq:KD-sharp-factor}\\
$t$ ($u$);\ $\alpha$ ($\beta$);\ $P$ ($Q$) & full second atom; floor
excess $t-m$; complement $1-t-nm$ (scalars of
Section~\ref{sec:D-core}; the carrier $Q$ of the introduction and
Section~\ref{sec:mul-reduction} does not occur there) &
\eqref{eq:D-parameters}\\
$\rho$ ($\sigma$);\ $\tau$;\ $\zeta$;\ $\pi$ & residual excess $R-P$;
balanced excess $\min(\beta\rho,\alpha\sigma)$; floor ratio
$\ell/\beta+m/\alpha$; combined atom $t+u-tu$ & \S\ref{sec:D-core}\\
$E_0$, $E_1$, $E_2$;\ $c_0$;\ $H_0$ & runner-up excess bounds; boundary
baseline; $H_0=c_0+E_0$ & \S\ref{sec:D-core}\\
\bottomrule
\end{tabular}
\caption{Shared notation.  Any symbol introduced inside a single lemma or
proof is local to it, and the appendices are notationally
self-contained.}
\label{tab:notation}
\end{table}

\subsection{The scalar function}\label{subsec:scalar}

Set $\delta_k(t):=t-\psi_k(t)$.  Direct differentiation gives
\begin{align}
  \psi_k'(t)&=\frac{k(1-t)^{k-1}-1}{k-1},
  &\psi_k''(t)&=-k(1-t)^{k-2},\label{eq:psi-derivatives}\\
  1-t-(k-1)\psi_k(t)&=(1-t)^k,
  &\psi_k(1-u)&=\frac{u-u^k}{k-1}.\label{eq:psi-identities}
\end{align}
Thus $\psi_k$ is concave, $0\le\psi_k(t)\le t$ on $[0,1]$, and
$\delta_k$ is convex and nondecreasing.  Concavity and $\psi_k(0)=0$ also
give
\begin{equation}\label{eq:star-shaped}
  t\psi_k'(t)\le\psi_k(t).
\end{equation}
The increasing branch of $\psi_k$ is $[0,t_k^*]$, where
\begin{equation}\label{eq:tstar}
  t_k^*:=1-k^{-1/(k-1)}.
\end{equation}

Define the cap point $\bar t_k\in(0,1)$ by
\begin{equation}\label{eq:cap-root}
  2\bar t_k+(k-2)\psi_k(\bar t_k)=1.
\end{equation}
The left side is strictly increasing, since
$\psi_k'\ge-1/(k-1)$.  For $k\ge4$ one has
\begin{equation}\label{eq:cap-increasing}
  \bar t_k\le t_k^*.
\end{equation}
Indeed, evaluation at $t_k^*$ reduces this to
$(k+2)^{k-1}\le k^k$.  It is immediate for $4\le k\le7$, and for
$k\ge8$ follows from $(1+2/k)^{k-1}<e^2<8\le k$.  The scalar estimates in
Appendices~\ref{app:mixed-sharp} and~\ref{app:D-reserve} abbreviate
$T:=\bar t_k$.

\subsection{Distinguished laws and nonvacuity}

\begin{proposition}[base laws]
\label{prop:B-base-laws}
The body $\body_k$ contains the uniform law and every point mass.
\end{proposition}

\begin{proof}
A point mass has $p_2^\down=p_k^\down=0$.  The uniform law satisfies
$p_k^\down=1/k\ge\psi_k(1/k)$ because $\psi_k(t)\le t$.
\end{proof}

The invariant is not merely nonempty: it occupies a full-dimensional region
of the probability simplex.  The following quantitative certificate is
uniform in $k$ and has a short machine-checkable proof.
The exact fraction and its sharp exponential scale are summarized immediately
afterward and proved in Appendix~\ref{app:exact-volume}.

\begin{proposition}[geometric nonvacuity]\label{prop:B-interior}
Let $\mathbf u=(1/k,\ldots,1/k)$ be the uniform law and put
\[
 m_k:=\frac1k-\psi_k(1/k)
     =\frac{(1-1/k)^k}{k-1}
     =\frac{(k-1)^{k-1}}{k^k}>0.
\]
Every probability vector $p$ satisfying
$\lVert p-\mathbf u\rVert_\infty<m_k/2$ belongs to $\body_k$.
Consequently $\mathbf u$ is a relative-interior point of $\body_k$ in the probability simplex.  More
quantitatively, if volume is normalized so that the whole $(k-1)$-simplex
has volume one, then
\begin{equation}\label{eq:B-volume-lower}
 \operatorname{vol}_{k-1}(\body_k)
 \ge \left(\frac{(k-1)^{k-1}}{2k^k}\right)^{k-1}>0.
\end{equation}
\end{proposition}

\begin{proof}
Formula~\eqref{eq:psi-derivatives} gives
$|\psi_k'(t)|\le1$ on $[0,1]$.  If
$\varepsilon=\lVert p-\mathbf u\rVert_\infty$, each sorted atom differs from $1/k$
by at most $\varepsilon$.  Hence
\[
 p_k^\down\ge\frac1k-\varepsilon,
 \qquad
 \psi_k(p_2^\down)\le\psi_k(1/k)+\varepsilon.
\]
The first quantity is larger than the second when $2\varepsilon<m_k$.

For the volume claim, set $\lambda=m_k/2$ and consider the homothetic simplex
$(1-\lambda)\mathbf u+\lambda\Delta_{k-1}$.  Every one of its points is
at sup-norm distance strictly less than $\lambda$ from $\mathbf u$, so
this simplex lies in
$\body_k$.  Homothety by $\lambda$ scales $(k-1)$-dimensional volume by
$\lambda^{k-1}$, giving \eqref{eq:B-volume-lower}.
\end{proof}

The normalized-volume assertion in \eqref{eq:B-volume-lower}, including its
strict positivity, is machine-checked independently of the gate theorems; see
Appendix~\ref{sec:lean}.  In particular, the formal result certifies the
relative volume fraction itself, not only the preceding set inclusion.

Writing
\[
 \mathfrak v_k:=
 \frac{\operatorname{vol}_{k-1}(\body_k)}
      {\operatorname{vol}_{k-1}(\Delta_{k-1})},
\]
Appendix~\ref{app:exact-volume} gives an exact one-dimensional integral and
proves
\[
 \lim_{k\to\infty}\frac{\log\mathfrak v_k}{k}=\Lambda_*,
 \qquad
 \lim_{k\to\infty}\mathfrak v_k^{1/k}
 =e^{\Lambda_*}=0.218330536985878\ldots.
\]
Here $\Lambda_*=-1.521745139795282\ldots$, with its exact saddle-point
definition given in that appendix.  Thus the elementary lower bound above
states nonvacuity for every $k\ge4$, while the exact calculation identifies
the fraction's sharp exponential scale.

The body condition itself supplies the cap used throughout the proof.

\begin{lemma}[member cap]\label{lem:member-cap}
If $p\in\body_k$, then $p_2^\down\le\bar t_k$.  Consequently $\psi_k$
is nondecreasing throughout the range of every second atom of a member.
\end{lemma}

\begin{proof}
Sorting and membership give
\[
  1\ge p_1^\down+p_2^\down+(k-2)p_k^\down
   \ge2p_2^\down+(k-2)\psi_k(p_2^\down).
\]
Use the strict increase in \eqref{eq:cap-root} and then
\eqref{eq:cap-increasing}.
\end{proof}

\subsection{The sharp family}\label{subsec:sharp}

The boundary family will appear repeatedly.  For $0\le u\le1$, let
\begin{equation}\label{eq:sharp-family}
 D_k(u):=\left(1-u,
 \frac{u+(k-2)u^k}{k-1},
 \underbrace{\frac{u-u^k}{k-1},\ldots,
 \frac{u-u^k}{k-1}}_{k-2\text{ entries}}\right).
\end{equation}
Its low atoms are $\psi_k(1-u)$, so after sorting it lies on the boundary of
$\body_k$ whenever its displayed middle coordinate is the largest atom: the
second atom is then $1-u$, and the smallest atom equals $\psi_k(1-u)$
exactly.

Define
\begin{equation}\label{eq:sharp-H}
 \mathcal H_k(t):=1-t-(k-2)\psi_k(t)
 =\psi_k(t)+(1-t)^k.
\end{equation}
This is the \emph{complementary atom} of the sharp profile: it completes
$t$ and $k-2$ copies of $\psi_k(t)$ to a law, and it is the largest atom of
that profile precisely when $\mathcal H_k(t)\ge t$.  In the coordinates of
\eqref{eq:sharp-family}, the middle entry of $D_k(u)$ is exactly
$\mathcal H_k(1-u)$.

\begin{lemma}[sharp-family inequality]\label{lem:sharp-family}
For $t\in[0,1]$,
\begin{equation}\label{eq:SF}
 \mathcal H_k(t)\le t
 \quad\Longrightarrow\quad
 \psi_k(\mathcal H_k(t))\le\psi_k(t).
\end{equation}
\end{lemma}

\begin{proof}
Put $u=1-t$, $H=\mathcal H_k(t)$, $v=1-H$, and
$\chi(x)=x-x^{n+1}$.  The desired inequality is $\chi(v)\le\chi(u)$,
while $H\le t$ is $v\ge u$.  Let $\alpha_0=(n+1)^{-1/n}=1-t_k^*$.  If
$u\ge\alpha_0$, then $\chi$ is decreasing on $[u,1]$, and the result is
immediate.

Suppose $0<u\le\alpha_0$; the case $u=0$ is trivial.  Put
\[
 W(x)=(1-x)^{n+1}-1+(n+1)x,\qquad
 G(u)=\frac{W(H(u))}{u^{n+1}}.
\]
The target is equivalent to $G(u)\ge n$, and it holds at $u=\alpha_0$ by
maximality of $\chi(\alpha_0)$.  With
$\theta=(n-1)u^{n}$, direct differentiation yields
\begin{equation}\label{eq:SF-derivative}
 G'(u)=-(n+1)u^{-n-2}\bigl(v^n(1+u\theta)-(1-u)\bigr).
\end{equation}
It remains to prove the bracket nonnegative.  Note
$\theta\le\theta_0:=(n-1)/(n+1)$ and
$v=1-u(1+\theta)/n$.  Taking logarithms reduces the sign to
\[
 L(u,\theta):=n\log\left(1-\frac{u(1+\theta)}n\right)
 +\log(1+u\theta)-\log(1-u)\ge0.
\]
For fixed $u$, $\partial_\theta L\le0$, hence
$L(u,\theta)\ge L(u,\theta_0)$.  If
$L_0(u)=L(u,\theta_0)$, then $L_0(0)=0$ and
\[
 L_0'(u)=\frac{2n(n-1)u^2}
 {(1-u)(n+1-2u)(n+1+(n-1)u)}\ge0.
\]
Thus the bracket in \eqref{eq:SF-derivative} is nonnegative, $G$ is
nonincreasing, and $G(u)\ge G(\alpha_0)\ge n$.
\end{proof}

\section{Quasigroup and addition closure for all
\texorpdfstring{$k\ge4$}{k >= 4}}
\label{sec:addition}

Let $Q$ have order $k\ge4$ and carry a quasigroup operation $\star$.
The field-addition case is obtained by taking $Q$ to be the additive group.
If the atom vectors of two laws are sorted as
\[
 x_1\ge\cdots\ge x_k,\qquad y_1\ge\cdots\ge y_k,
\]
define the anti-aligned and runner-up pairings
\begin{equation}\label{eq:add-AB}
 A(x,y):=\sum_{i=1}^k x_i y_{k+1-i},\qquad
 B(x,y):=x_1y_2+x_2y_1+\sum_{i=3}^k x_i y_i.
\end{equation}

\begin{lemma}[quasigroup rearrangement]\label{lem:add-rearrangement}
If $p$ is their convolution with respect to $\star$, then
\[
 p_k^\down\ge A(x,y),\qquad p_2^\down\le B(x,y).
\]
\end{lemma}

\begin{proof}
Enumerate $Q$ so that the first law has atoms $x_i$ and the second has atoms
$y_j$.  For a fixed output and each left input, the relevant right input is
the unique solution of the corresponding quasigroup equation.  As the left
input varies, these solutions form a permutation, so every convolution atom
is $\sum_i x_i y_{\pi(i)}$ for a permutation $\pi$.  The rearrangement
inequality \cite[Theorem~368]{hardy1952} gives the first bound.

For the second, suppose $\pi(1)\ne1$.  If $\pi(1)\ne2$, exchange the
partners of position $1$ and $\pi^{-1}(2)$; the scalar product increases by
$(x_1-x_{\pi^{-1}(2)})(y_2-y_{\pi(1)})\ge0$.  Then exchange the partner
of position $2$ with that of $\pi^{-1}(1)$ if necessary.  The remaining
coordinates are maximized by aligned pairing, so the product is at most
$B(x,y)$.  Only one output element can pair the carrier elements supporting
a selected $x_1$ and $y_1$.  Hence all but at most one convolution atom are
at most $B$, which proves the second-largest bound.
\end{proof}

The same cap that controls a member also controls the runner-up pairing.

\begin{lemma}[pairing cap]\label{lem:add-pairing-cap}
For $x,y\in\body_k$, one has $B(x,y)\le\bar t_k$.
\end{lemma}

\begin{proof}
For a member $y$, sorting and membership give
$y_1\le\mathcal H_k(y_2)$.  Using $y_i\le y_2$ for $i\ge3$ and
$x_2\le\bar t_k$,
\[
 B\le y_2+x_2(y_1-y_2)
 \le y_2+\bar t_k\bigl(\mathcal H_k(y_2)-y_2\bigr).
\]
The right side is convex in $y_2$ on $[0,\bar t_k]$ and equals
$\bar t_k$ at both endpoints; therefore it is at most $\bar t_k$
throughout.
\end{proof}

We now prove the scalar comparison $A\ge\psi_k(B)$.  For a sorted member
$w$, define the coefficients that arise when $\psi_k$ is replaced by its
tangent line at $w_2$ (Lemma~\ref{lem:add-tangent} below): $T_w$ will
multiply the partner's second atom, the coefficients $C_i(w)$ its lower
atoms, and $\Sigma C(w)$ is their total:
\begin{align}
 T_w&=(w_{k-1}-w_k)
      -\psi_k'(w_2)(w_1-w_2),\label{eq:add-K}\\
 C_i(w)&=(w_{k+1-i}-w_k)
      +\psi_k'(w_2)(w_2-w_i),
      \qquad3\le i\le k,\label{eq:add-C}\\
 \Sigma C(w)&=\sum_{i=3}^kC_i(w)=1-w_{k-1}-(k-1)w_k
 +\psi_k'(w_2)
   \bigl((k-1)w_2+w_1-1\bigr).\label{eq:add-sumC}
\end{align}
Every $C_i(w)$ is nonnegative.

\begin{lemma}[tangent bounds]\label{lem:add-tangent}
For members $x,y$,
\begin{align*}
 A(x,y)-\psi_k(B(x,y))&\ge
 U_1:=x_2T_y+\sum_{i=3}^k x_iC_i(y),\\
 A(x,y)-\psi_k(B(x,y))&\ge
 U_2:=y_2T_x+\sum_{i=3}^k y_iC_i(x).
\end{align*}
\end{lemma}

\begin{proof}
Concavity gives
$\psi_k(B)\le\psi_k(y_2)+\psi_k'(y_2)(B-y_2)
\le y_k+\psi_k'(y_2)(B-y_2)$.
Expanding $A-y_k-\psi_k'(y_2)(B-y_2)$ by coordinates of $x$ gives
exactly $U_1$; the coefficient of $x_1$ vanishes.
Interchanging $x,y$ gives $U_2$.
\end{proof}

\begin{lemma}[single-member inequality]\label{lem:add-single}
Every sorted member $w$ satisfies
\begin{equation}\label{eq:add-single}
 F(w):=w_2T_w+\psi_k(w_2)\Sigma C(w)\ge0.
\end{equation}
\end{lemma}

\begin{proof}
Fix $t=w_2$, put $\psi=\psi_k(t)$ and $d=\psi_k'(t)$, and make the
dependence on the remaining displayed coordinates explicit:
\[
 \Phi_t(u,v,z):=
 t\bigl[(v-z)-d(u-t)\bigr]
 +\psi\bigl[1-v-(k-1)z+d((k-1)t+u-1)\bigr].
\]
Thus $F(w)=\Phi_t(w_1,w_{k-1},w_k)$.  First hold
$t,v=w_{k-1},z=w_k$ fixed.  Lower each of
$w_3,\ldots,w_{k-2}$ to $v$ and transfer the removed mass to
$u=w_1$.  The middle coordinates do not otherwise occur in $\Phi_t$, and
\[
 \frac{\partial\Phi_t}{\partial u}=d(\psi-t)\le0,
\]
so this mass transfer can only decrease $F$.

Let $\one=(1,\ldots,1)\in\R^{k-3}$.  After this first move the vector has
the form
\[
 (u,t,v\one,z),
 \qquad u=1-t-(k-3)v-z.
\]
Now hold $t$ and $z$ fixed, lower every component of the block $v\one$
together, and again transfer the removed mass to $u$.  In other words,
restrict $\Phi_t$ to the mass-preserving one-variable path
\[
 \widetilde\Phi_{t,z}(v)
 :=\Phi_t\bigl(1-t-(k-3)v-z,v,z\bigr).
\]
Its derivative is
\[
 \frac{\mathrm d}{\mathrm dv}\widetilde\Phi_{t,z}(v)
 =\frac{\partial\Phi_t}{\partial v}
  -(k-3)\frac{\partial\Phi_t}{\partial u}
 =(t-\psi)\bigl(1+(k-3)d\bigr)\ge0.
\]
Hence lowering $v$ to $z$ also decreases $F$.  Both moves preserve total
mass, sorting, $w_2=t$, and $w_k=z$, so membership is preserved.  At the
resulting vector $(1-t-(k-2)z,t,z\one,z)$, the expression is linear in
$z$:
\begin{equation}\label{eq:add-z-segment}
 F=-\psi_k'(t)t\bigl(1-2t-(k-2)z\bigr)
 +\psi\bigl[1-kz+(k-2)\psi_k'(t)(t-z)\bigr].
\end{equation}
Membership and sorting place
\[
 z\in\left[\psi_k(t),\min\left(t,\frac{1-2t}{k-2}\right)\right].
\]
At the two possible upper endpoints, \eqref{eq:add-z-segment} becomes
\[
 (1-kt)\bigl(\psi-t\psi_k'(t)\bigr)\ge0\quad(kt\le1),
 \qquad
 -\psi(1-kt)\left(\frac2{k-2}+\psi_k'(t)\right)\ge0
 \quad(kt\ge1).
\]
At the lower endpoint $z=\psi_k(t)$ it becomes the one-variable margin in
Lemma~\ref{lem:L1ss} below.  Linearity closes the entire segment.
\end{proof}

\begin{lemma}[crossing]\label{lem:add-crossing}
For all members $x,y$, $\max(U_1,U_2)\ge0$.  Consequently
$A(x,y)\ge\psi_k(B(x,y))$.
\end{lemma}

\begin{proof}
Suppose $U_1,U_2<0$.  Then $x_2,y_2>0$ and $T_x,T_y<0$.  Since every
atom of $x$ is at least $\psi_k(x_2)$,
\[
 \frac{x_2}{\psi_k(x_2)}>
 \frac{\Sigma C(y)}{-T_y}\ge
 \frac{y_2}{\psi_k(y_2)},
\]
where the second inequality is Lemma~\ref{lem:add-single} for $y$.  The
symmetric argument gives the strict reverse inequality, a contradiction.
The conclusion follows from Lemma~\ref{lem:add-tangent}.
\end{proof}

It remains to establish the one-variable endpoint used above.  The proof is
included because it is the only scalar input to the addition gate.  The
lemma is stated for real $k$ because integrality plays no role in it; we
apply it only for integer orders.

\begin{lemma}[polynomial-band inequality]\label{lem:L1ss}
Let $k\ge4$ be real and let $t\in[0,1]$ satisfy
$2t+(k-2)\psi_k(t)\le1$.  With $\psi=\psi_k(t)$,
\begin{equation}\label{eq:L1ss}
 \psi_k'(t)\bigl[t(2t-1)+(k-2)\psi(2t-\psi)\bigr]
 +\psi(1-k\psi)\ge0.
\end{equation}
\end{lemma}

\begin{proof}
Write $s=1-t$ and $\omega=s^{k-1}$.  After multiplying the left
side of \eqref{eq:L1ss} by $(k-1)^3$, direct expansion gives
\begin{align}
 \mathcal P(t,\omega)={}&(k-1)^2(k\omega-1)t(2t-1)\notag\\
 &+(k-2)(k\omega-1)s(1-\omega)\bigl(2(k-1)t-s(1-\omega)\bigr)\notag\\
 &+(k-1)s(1-\omega)\bigl((k-1)-ks(1-\omega)\bigr).
 \label{eq:L1-P}
\end{align}
Bernoulli's inequality and the cap hypothesis place $\omega$ in the
polynomial band
\begin{equation}\label{eq:L1-band}
 \max(\omega_L,\omega_B)\le \omega\le \omega_U,
\end{equation}
where
\[
 \omega_L=1-(k-1)t,\quad
 \omega_U=\frac{s}{s+(k-1)t},\quad
 \omega_B=1-\frac{(k-1)(2s-1)}{(k-2)s}
\]
(the subscripts stand for the lower and upper edges and for the edge
active in the bulk region below).  The band forces $0\le t\le1/2$.  Put
$\gamma=(k-2)t^2-kt+1$.  Since
\[
 \omega_L-\omega_B=\frac{(k-1)\gamma}{(k-2)s},
\]
the lower edge is $\omega_L$ when $\gamma\ge0$ and $\omega_B$ when
$\gamma\le0$.

The polynomial $\mathcal P$ is cubic in $\omega$ with negative leading
coefficient, and $\mathcal P_{\omega\omega}$ decreases in $\omega$.  Thus
it is enough to establish concavity at the lower edge and nonnegativity
at the two edges.  The needed signs follow from the following five exact
identities:
\begin{align}
 \mathcal P(t,\omega_B)((k-2)s)^3
 &=-k(k-2)(k-1)^3t(t-1)^2(2t-1)(kt-1),\label{eq:L1-D}\\
 -\mathcal P_{\omega\omega}(t,\omega_B)(k-2)s
 &=2(k-2)(k-1)(t-1)^2(2k^2t-k+2t-2),\label{eq:L1-concD}\\
 \mathcal P(t,\omega_U)(s+(k-1)t)^3
 &=-t^2(k-1)^3G_U,\label{eq:L1-U}\\
 \mathcal P(t,\omega_L)&=t^2(k-1)^3G_L,\label{eq:L1-L}\\
 -\mathcal P_{\omega\omega}(t,\omega_L)&=2(k-1)(1-t)G_C,\label{eq:L1-concL}
\end{align}
where
\begin{align*}
 -G_U={}&kt\bigl(k(1-t-t^2)-2t(1-2t)\bigr)
 +(1-2t)((1-t)^2+t^2),\\
 (1-t)G_L={}&\gamma\bigl(k(1-t^2)-1\bigr)+t(1-2t^2),\\
 G_C={}&3k\gamma+2(k^2+1)t-(k+2).
\end{align*}
On the bulk region $\gamma\le0$, one has $kt\ge1$, so
\eqref{eq:L1-D} and \eqref{eq:L1-concD} have the required signs.  The
upper identity is nonnegative throughout $[0,1/2]$, since
$1-t-t^2\ge1/4$ and $2t(1-2t)\le1/4$.

On the corner region $\gamma\ge0$, \eqref{eq:L1-L} is nonnegative by
direct sign inspection.  For $G_C$, put
$t_0=(k+2)/(2(k^2+1))$.  If $t\ge t_0$, both terms in its displayed
decomposition are nonnegative.  If $t\le t_0\le1/k$, then
\[
 G_C\ge2(k-1)-(k^2-2)t\ge0;
\]
the last inequality is equivalent to
$3k(k^2-2k+2)\ge0$.  Hence $\mathcal P$ is concave across the band and
nonnegative at both endpoints in either region.  Therefore
$\mathcal P(t,\omega)\ge0$, proving \eqref{eq:L1ss}.
\end{proof}

\begin{theorem}[ordinary quasigroup gate]\label{thm:addition}
For every quasigroup operation on a set of order $k\ge4$, convolution
preserves $\body_k$.  In particular, additive convolution preserves
$\body_k$ on every finite abelian group of order $k$.
\end{theorem}

\begin{proof}
For two members with quasigroup convolution $p$,
Lemmas~\ref{lem:add-rearrangement} and
\ref{lem:add-crossing} give
\[
 p_k^\down\ge A\ge\psi_k(B).
\]
By Lemma~\ref{lem:add-pairing-cap}, both $p_2^\down$ and $B$ lie on the
increasing branch of $\psi_k$, so
$\psi_k(p_2^\down)\le\psi_k(B)\le p_k^\down$.
\end{proof}

\section{Multiplication: structural reduction}\label{sec:mul-reduction}

We turn to multiplication.  For the rest of
Sections~\ref{sec:mul-reduction}--\ref{sec:assembly}, let $Q$ have order
$k\ge4$ and carry an absorbing-zero quasigroup operation $\diamond$.
Thus $Q^\times=Q\setminus\{0\}$ is a quasigroup of order $n=k-1$.
Finite-field multiplication is the motivating special case.

Let $\mu\in\body_k$.  Separate its zero atom and nonzero block by writing
\begin{equation}\label{eq:floor-residual-one}
 a=\mu(0),\qquad r=1-a,\qquad
 m=p_k^\down(\mu),\qquad x=m\one+\hat x,\qquad R=\sum \hat x.
\end{equation}
Here $\one=(1,\ldots,1)\in\R^n$ is the all-ones vector indexed by
$Q^\times$, and $m$ is the \emph{global} smallest atom; in particular,
it need not be the minimum of the nonzero block when zero itself is smallest. 
The vector $x$ is the nonzero block of $\mu$: a subprobability vector on
$Q^\times$ of total mass $r$, read in the carrier labeling in
\eqref{eq:floor-residual-one} and \eqref{eq:nonzero-convolution}, and
identified with its sorted version whenever only the multiset of its
entries matters.
Use $b,s,\ell,y=\ell\one+\hat y,S$ for a second law $\nu$. 
The product has zero mass and nonzero common baseline
\begin{equation}\label{eq:mul-baseline}
 z=a+b-ab=1-rs,\qquad
 c=nm\ell+mS+\ell R=\frac{rs-RS}{n}.
\end{equation}
Indeed, on $Q^\times$ the product law is
\begin{equation}\label{eq:nonzero-convolution}
  w=c\one+\hat x\mathbin{*}\hat y,
\end{equation}
where $*$ is quasigroup convolution with respect to $\diamond$.

\subsection{The K/D dichotomy}

Call an input \emph{K-type} if its zero atom is largest and \emph{D-type}
otherwise; mnemonically, zero \emph{K}eeps the top in the first case, while
in the second the mass is \emph{D}ispersed among the nonzero elements.  If
$p=\max_{g\ne0}\mu(g)$, membership in $\body_k$ gives
\begin{equation}\label{eq:KD-dichotomy}
 \begin{array}{ll}
 \mathrm{K}:&m\ge\psi_k(p),\\[2pt]
 \mathrm{D}:&R\le r^k.
 \end{array}
\end{equation}
The K-line is immediate because $p$ is the second atom.  In the D-line the
second atom is at least $a$, so cap monotonicity gives
$m\ge\psi_k(a)=(r-r^k)/n$, hence $R=r-nm\le r^k$.

If either input is K-type, the product zero atom is largest.  For example,
if $a\ge\mu(g)$ for every $g\ne0$, then every nonzero product atom obeys
\[
 (\mu\otimes\nu)(g)
 =\sum_{g'\ne0}\mu(g')\nu\bigl(\sigma_g(g')\bigr)
 \le as\le a+b-ab=z,
\]
where $\sigma_g(g')$ is the unique nonzero solution $g''$ of
$g'\diamond g''=g$.  For fixed $g$, the map $g'\mapsto\sigma_g(g')$ is a
permutation of $Q^\times$, which gives the displayed sum bound.

Sort the two nonzero blocks as $x_1\ge\cdots\ge x_n$ and
$y_1\ge\cdots\ge y_n$, and define
\begin{align}
 A&=\sum_{i=1}^n x_i y_{n+1-i},
 &C&=\sum_{i=1}^n x_i y_i,\label{eq:mul-AC}\\
 B&=x_1y_2+x_2y_1+\sum_{i=3}^n x_i y_i.\label{eq:mul-B}
\end{align}

\begin{lemma}[multiplicative rearrangement]\label{lem:mul-rearrangement}
For the nonzero atoms of $\mu\otimes\nu$,
\[
 \min_{g\ne0}w_g\ge A,\qquad
 \max_{g\ne0}w_g\le C,
 \qquad
 \operatorname{second}_{g\ne0}w_g\le B,
\]
where $\operatorname{second}$ is the second largest value, counted with
multiplicity.
\end{lemma}

\begin{proof}
Every $w_g$ is a scalar product of $x$ with a permutation of $y$.  The
first two bounds are the two directions of the rearrangement
inequality~\cite[Theorem~368]{hardy1952}.
At most one output element pairs fixed choices of $x_1$ and $y_1$; the same two
exchanges used in Lemma~\ref{lem:add-rearrangement} bound every other scalar
product by $B$.
\end{proof}

The cap of Lemma~\ref{lem:add-pairing-cap} extends to the nonzero
blocks: the atom missing from them is zero, and the zero atom also
dominates the global floor.

\begin{lemma}[multiplicative pairing cap]\label{lem:mul-pairing-cap}
For $\mu,\nu\in\body_k$, the pairing \eqref{eq:mul-B} of their sorted
nonzero blocks satisfies $B\le\bar t_k$.
\end{lemma}

\begin{proof}
A second nonzero atom is at most the full second atom, so
Lemma~\ref{lem:member-cap} gives $x_2\le\bar t_k$ and
$y_2\le p_2^\down(\nu)\le\bar t_k$.  The $k-2$ atoms of $\nu$ other than
$y_1,y_2$ --- the zero atom among them --- are each at least
$p_k^\down(\nu)\ge\psi_k\bigl(p_2^\down(\nu)\bigr)\ge\psi_k(y_2)$, by
membership and cap monotonicity.  Total mass therefore gives
$y_1\le1-y_2-(k-2)\psi_k(y_2)=\mathcal H_k(y_2)$.  Bounding the tail of
$B$ by $y_2\sum_{i\ge3}x_i$ and the mass of the first block by one,
\[
 B\le y_2+x_2(y_1-y_2)
 \le y_2+\bar t_k\bigl(\mathcal H_k(y_2)-y_2\bigr).
\]
As in Lemma~\ref{lem:add-pairing-cap}, the right side is convex in
$y_2$ on $[0,\bar t_k]$ and equals $\bar t_k$ at both endpoints
($\mathcal H_k(0)=1$; $\mathcal H_k(\bar t_k)=\bar t_k$ by
\eqref{eq:cap-root}), so $B\le\bar t_k$.
\end{proof}

\subsection{What remains after rearrangement}

For two D-type inputs, \eqref{eq:KD-dichotomy} and
\eqref{eq:mul-baseline} imply
\begin{equation}\label{eq:DD-baseline}
 A\ge c\ge\frac{rs-(rs)^k}{n}=\psi_k(z).
\end{equation}
They also imply
\begin{equation}\label{eq:z-ge-c}
 z\ge c.
\end{equation}
Indeed, this is equivalent to $RS\ge krs-(k-1)$.  It is automatic when
$rs\le(k-1)/k$.  Otherwise $r,s>(k-1)/k$ and, because the global floor is
at most the zero atom,
\[
 R\ge kr-(k-1),\qquad S\ge ks-(k-1).
\]
The product of the right sides exceeds $krs-(k-1)$ by
$k(k-1)(1-r)(1-s)\ge0$.

The multiplication gate is now reduced to the following scalar statements:
\begin{align}
 \text{K-core:}\quad&A\ge\psi_k(C)
 &&\text{if at least one input is K-type},\label{eq:K-core}\\
 \text{D-core:}\quad&c\ge\psi_k(B)
 &&\text{for two D-type inputs in the branch }B>z.
 \label{eq:D-core}
\end{align}
Sections~\ref{sec:K-core} and \ref{sec:D-core} prove them.

\section{The K-core}\label{sec:K-core}

We now prove \eqref{eq:K-core}.  The case in which both inputs are K-type
contains the central rank inequality; the mixed K$\times$D case is then
compared to a sharp boundary factor.

\subsection{\texorpdfstring{K$\times$K}{K x K}: floor and residual reduction}

Write the sorted nonzero blocks as
\[
 x_1\ge\cdots\ge x_n,\qquad y_1\ge\cdots\ge y_n,
\]
so that $p=x_1$ and $v=y_1$.  Both zero atoms are largest, so the global
floors are attained on the blocks: $m=x_n$, $\ell=y_n$, and hence
$U=p-m$ and $R=\sum_i(x_i-m)$, with $V,S$ the analogous quantities for
the second block.  K-type membership gives
\begin{equation}\label{eq:KK-hypotheses}
 m\ge\psi_k(p),\qquad \ell\ge\psi_k(v),\qquad
 1-\sum_i x_i\ge p,\qquad1-\sum_i y_i\ge v.
\end{equation}
The floor decomposition gives
\begin{equation}\label{eq:KK-floor}
 A\ge c,\qquad C\le c+E,\qquad
 E:=\min(RV,SU).
\end{equation}
Put $\Theta=c+E$.  If $E=SU$, then
\[
 \Theta=\ell(nm+R)+pS\le p(n\ell+S)\le p;
\]
the other branch gives $\Theta\le v$.  Thus $\Theta$ lies in the cap and
it suffices to prove $c\ge\psi_k(\Theta)$.

If $U=0$ or $V=0$, the conclusion follows directly from the floor
condition.  Otherwise define the effective rank
\[
 \lambda=\min(R/U,\,S/V)\in[1,n].
\]
Lower the larger residual excess.  This subtracts the same nonnegative
quantity from $c$ and $\Theta$; because $z\mapsto z-\psi_k(z)$ is
increasing, it can only decrease the margin.  We may therefore assume
\begin{equation}\label{eq:KK-diagonal}
 R=\lambda U,\qquad S=\lambda V.
\end{equation}
Let
\[
 a_0=m/p,\qquad b_0=\ell/v,\qquad
 \kappa=\lambda+(n-\lambda)a_0b_0.
\]
Then exact expansion gives
\begin{equation}\label{eq:KK-T}
 \Theta=pv\kappa,\qquad \Theta-c=\lambda pv(1-a_0)(1-b_0).
\end{equation}

Define continuously at zero
\begin{equation}\label{eq:h-def}
 h(t):=\frac{\delta_k(t)}{t^2},\qquad h(0)=k/2.
\end{equation}
The K-floor inequalities imply
\begin{equation}\label{eq:h-floors}
 \frac{1-a_0}{p}\le h(p),\qquad
 \frac{1-b_0}{v}\le h(v).
\end{equation}

\begin{lemma}[multiplicative $h$ inequality]\label{lem:h-multiplicative}
If $M\ge k/2$, $t_1,t_2\in[0,1]$, and $Mt_1t_2\le1$, then
\begin{equation}\label{eq:h-multiplicative}
 M h(Mt_1t_2)\ge h(t_1)h(t_2).
\end{equation}
\end{lemma}

\begin{proof}
From the finite-sum representation
\[
 h(t)=\frac1n\sum_{j=0}^{n-1}(n-j)(1-t)^j
\]
one obtains
\[
 \left(t\frac{\mathrm d}{\mathrm d t}\right)^2\log h(t)
 =-\frac{(n+1)tP_n(1-t)}{n^2h(t)^2}\le0,
\]
where
\[
 P_n(z)=\sum_{j=0}^{2n-4}
 \binom{\min(j,2n-4-j)+3}{3}z^j
\]
has positive coefficients.  Thus $u\mapsto\log h(e^u)$ is concave, and
for fixed $t_1t_2$ the product $h(t_1)h(t_2)$ is maximized at
$t_1=t_2=\sqrt{t_1t_2}$.  It remains to prove
$h(t)^2\le M h(Mt^2)$ for $0\le t\le M^{-1/2}$.  The logarithmic
slope $-\mathrm d\log h(t)/\mathrm d\log t$ is increasing by the same identity, so the
quotient $h(t)^2/(M h(Mt^2))$ decreases on $[0,1/M]$ and increases
on $[1/M,1/\sqrt M]$.  At $t=0$ it is $(k/2)/M\le1$; at the other
endpoint it is at most one because $h(1)=1$ and $h(t)\le1/t$.  This
proves \eqref{eq:h-multiplicative}.
\end{proof}

The remaining input is the rank inequality.

\begin{lemma}[K-rank inequality]\label{lem:K-rank}
Under \eqref{eq:KK-hypotheses} and \eqref{eq:KK-diagonal},
\begin{equation}\label{eq:K-rank}
 \kappa^2\ge\lambda k/2.
\end{equation}
\end{lemma}

\begin{proof}
The reduction is short, while the two endpoint certificates are collected in
Appendix~\ref{app:rank}.  Put $\varphi(t)=\psi_k(t)/t$ continuously at
zero and assume $p\ge v$.  Since $\varphi$ decreases,
\[
 \kappa\ge J:=\lambda+(n-\lambda)\varphi(p)^2.
\]
The mass and largest-zero constraints give
\begin{equation}\label{eq:rank-domain}
 (n-\lambda)\psi_k(p)+(\lambda+1)p\le1.
\end{equation}
For fixed $\lambda$, $J$ decreases with $p$, so push to equality in
\eqref{eq:rank-domain}.  With $\xi=1/p$ and $\varphi=\varphi(p)$ this gives
\begin{equation}\label{eq:rank-identities}
 \xi=\lambda+1+(n-\lambda)\varphi,\qquad
 \lambda(1-\varphi)=\xi(1-1/\xi)^k.
\end{equation}
A hypothetical $J<\sqrt{\lambda k/2}$ is contradicted by the boundary
inequality
\begin{equation}\label{eq:rank-boundary-target}
 \xi_0(1-1/\xi_0)^k<\lambda(1-\varphi_0),
\end{equation}
where
\[
 \varphi_0^2=\frac{\sqrt{\lambda k/2}-\lambda}{n-\lambda},\qquad
 \xi_0=\lambda+1+(n-\lambda)\varphi_0.
\]
Indeed, $J<\sqrt{\lambda k/2}$ forces $\varphi<\varphi_0$, hence
$\xi<\xi_0$ by the first identity in \eqref{eq:rank-identities}.  Since
$\xi\mapsto\xi(1-1/\xi)^k$ increases for $\xi>1$ (its logarithmic
derivative is $(\xi+k-1)/(\xi(\xi-1))$), the second identity and
\eqref{eq:rank-boundary-target} give
\[
 \lambda(1-\varphi)=\xi(1-1/\xi)^k\le\xi_0(1-1/\xi_0)^k
 <\lambda(1-\varphi_0)<\lambda(1-\varphi),
\]
a contradiction.
Appendix~\ref{app:rank} proves \eqref{eq:rank-boundary-target} uniformly
for $1\le\lambda<k/2$; the case $\lambda\ge k/2$ is immediate from
$\kappa\ge\lambda$.
\end{proof}

\begin{theorem}[K$\times$K core]\label{thm:KK-core}
For two K-type members, $A\ge\psi_k(C)$.
\end{theorem}

\begin{proof}
By Lemma~\ref{lem:K-rank}, set $M=\kappa^2/\lambda\ge k/2$ and apply
Lemma~\ref{lem:h-multiplicative} with $t_1=p$ and
$t_2=v\lambda/\kappa$.  Since $\kappa\ge\lambda$, $t_2\le v$, and
$h$ is decreasing.  Therefore
\[
 \frac{\kappa^2}{\lambda}h(\Theta)
 \ge h(p)h(v\lambda/\kappa)\ge h(p)h(v).
\]
Together with \eqref{eq:h-floors}, this is
\[
 \kappa^2h(\Theta)\ge
 \lambda\frac{1-a_0}{p}\frac{1-b_0}{v},
\]
which is equivalent, by \eqref{eq:KK-T}, to $c\ge\psi_k(\Theta)$.  Now
\eqref{eq:KK-floor} and cap monotonicity give
$A\ge c\ge\psi_k(\Theta)\ge\psi_k(C)$.
\end{proof}

\subsection{The sharp mixed scalar}

\begin{lemma}[sharp mixed inequality]\label{lem:mixed-sharp}
Let $0\le p\le\bar t_k$ and $0\le t\le1$.  Then
\begin{equation}\label{eq:mixed-sharp}
 \delta_k\bigl((1-t)p-(k-2)\psi_k(t)\delta_k(p)\bigr)
 \ge(1-t)^k\delta_k(p).
\end{equation}
\end{lemma}

\begin{proof}
When $t\le\bar t_k$, the analytic proof is given in
Appendix~\ref{app:mixed-sharp}.  For $t\ge\bar t_k$, compare the two sharp
blocks
\[
 (p,\psi_k(p),\ldots,\psi_k(p)),\qquad
 (\psi_k(t)+(1-t)^k,\psi_k(t),\ldots,\psi_k(t)).
\]
Their zero atoms are respectively
$1-p-(n-1)\psi_k(p)=\mathcal H_k(p)$ and $t$; as atom multisets, the two
laws are exactly the sharp profiles $D_k(1-p)$ and $D_k(1-t)$ of
Section~\ref{subsec:sharp}.  The first block is K-type by the cap; the
second is K-type by Lemma~\ref{lem:sharp-family}.  Its aligned and
anti-aligned pairings differ by exactly $(1-t)^k\delta_k(p)$, so
Theorem~\ref{thm:KK-core} is precisely \eqref{eq:mixed-sharp}.
\end{proof}

\subsection{\texorpdfstring{K$\times$D}{K x D}}

Let the K input have largest nonzero atom $p=m+U$, residual total $R$, and
let the D input have global floor $\ell$, residual maximum $V$, residual total
$S$, and nonzero mass
\[
 s=n\ell+S.
\]
The D cap is $S\le s^k$.  As in \eqref{eq:KK-floor},
\begin{equation}\label{eq:KD-floor}
 A\ge c,\qquad C\le c+E,\qquad
 c=nm\ell+mS+\ell R,\qquad E=\min(RV,SU).
\end{equation}
Put $\Theta=c+E$.  Since $E\le SU$,
\begin{equation}\label{eq:KD-cap}
 \Theta\le \ell(nm+R)+pS\le p(n\ell+S)=ps\le p\le\bar t_k.
\end{equation}
It remains to prove $\delta_k(\Theta)\ge E$.

If $E=0$ there is nothing to prove.  Otherwise $U,V>0$; define the
effective rank and the two removals
\begin{equation}\label{eq:KD-parameters}
 \lambda=\frac E{UV},\qquad
 \Gamma=R-\lambda U,\qquad\Gamma'=S-\lambda V.
\end{equation}
Then $\lambda\ge1$ and $\Gamma,\Gamma'\ge0$.  Remove only $\Gamma$ from
the first residual total and set
\begin{equation}\label{eq:KD-T0}
 \Theta_0=nm\ell+mS+\ell\lambda U+E=ms+\lambda U(\ell+V),\qquad
 \Theta=\Theta_0+\ell\Gamma.
\end{equation}
Let
\begin{equation}\label{eq:KD-sharp-factor}
 s_0=S^{1/k},\qquad
 \eta=\psi_k(1-s_0)=\frac{s_0-S}{n}.
\end{equation}
Since $S\le s^k$, $s_0\le s$ and hence $\ell=(s-S)/n\ge\eta$.
Set
\begin{equation}\label{eq:KD-comparison-points}
 x_U=ps_0-(n-1)\eta U,\qquad
 x_\delta=ps_0-(n-1)\eta\,\delta_k(p).
\end{equation}
K membership gives $U\le \delta_k(p)$, so $x_U\ge x_\delta$.  All of
$\Theta_0,x_U,x_\delta,\Theta$ lie in $[0,p]$.  Exact algebra, using
$s_0=S+n\eta$, gives
\begin{equation}\label{eq:KD-cancellation}
 \Theta_0-x_U=mn(\ell-\eta)+U(\lambda \ell-\eta-\Gamma')\ge-U\Gamma'.
\end{equation}
On the cap, $\delta_k$ is increasing and $1$-Lipschitz.  Therefore
Lemma~\ref{lem:mixed-sharp}, applied at $t=1-s_0$ so that $1-t=s_0$,
$\psi_k(t)=\eta$, and $(1-t)^k=S$, yields
\begin{align*}
 \delta_k(\Theta)&\ge\delta_k(\Theta_0)
 \ge\delta_k(x_U)-U\Gamma'
 \ge\delta_k(x_\delta)-U\Gamma'\\
 &\ge S\,\delta_k(p)-U\Gamma'\ge SU-U\Gamma'=\lambda UV=E.
\end{align*}
We have proved the mixed core.

\begin{theorem}[K-core]\label{thm:K-core}
If at least one of two members of $\body_k$ is K-type, then
$A\ge\psi_k(C)$ and $C\le\bar t_k$.
\end{theorem}

\begin{proof}
The K$\times$K case is Theorem~\ref{thm:KK-core}.  The K$\times$D case
follows from \eqref{eq:KD-floor}, \eqref{eq:KD-cap}, and the preceding
calculation; D$\times$K is symmetric.  The cap on $C$ was established
along the way in both cases: $C\le\Theta$, with $\Theta\le\max(p,v)\le
\bar t_k$ in the K$\times$K case (display after \eqref{eq:KK-floor};
$p$ and $v$ are second atoms of members) and $\Theta\le p\le\bar t_k$
in the mixed case \eqref{eq:KD-cap}.
\end{proof}

\section{The D-core}\label{sec:D-core}

Assume both inputs are D-type.  Let $t,u$ be their full second atoms and
$m,\ell$ their global floors.  Put
\begin{equation}\label{eq:D-parameters}
 \alpha=t-m,\qquad\beta=u-\ell,\qquad
 P=1-t-nm,\qquad Q=1-u-n\ell.
\end{equation}
Sorting and total mass give, by the estimate in the proof of
Lemma~\ref{lem:member-cap} without its membership substitution,
\begin{equation}\label{eq:D-signs}
 P-\alpha=1-2t-(k-2)m\ge0,\qquad
 Q-\beta=1-2u-(k-2)\ell\ge0,
\end{equation}
and $\alpha,\beta\ge0$; in particular $0\le\alpha\le P$ and
$0\le\beta\le Q$.
For the residual blocks in \eqref{eq:floor-residual-one},
\[
 R\ge P,\quad \hat x_1\le P,\quad \hat x_2\le\alpha,\qquad
 S\ge Q,\quad \hat y_1\le Q,\quad \hat y_2\le\beta.
\]
The runner-up residual pairing consequently satisfies
\begin{equation}\label{eq:D-runner}
 B-c\le\min(E_1,E_2),
\end{equation}
where
\begin{equation}\label{eq:D-E12}
 E_1=R\beta+\alpha(Q-\beta),\qquad
 E_2=S\alpha+\beta(P-\alpha).
\end{equation}
Indeed,
$\hat x_1\hat y_2+\hat x_2\hat y_1+\sum_{i\ge3}\hat x_i\hat y_i
\le\hat x_1\beta+\hat x_2Q+\beta(R-\hat x_1-\hat x_2)
=R\beta+\hat x_2(Q-\beta)\le E_1$,
using $\hat y_i\le\hat y_2\le\beta$ for $i\ge3$, then $\hat x_2\le\alpha$
and $Q\ge\beta$ from \eqref{eq:D-signs}; the second formula is the
symmetric estimate.

Write
\begin{equation}\label{eq:D-excess}
 R=P+\rho,\qquad S=Q+\sigma,\qquad
 E_0=P\beta+\alpha Q-\alpha\beta,\qquad
 c_0=nm\ell+mQ+\ell P.
\end{equation}
By \eqref{eq:D-signs}, all of $E_0$, $c_0$, and $H_0:=c_0+E_0$ are
nonnegative.  Moreover $E_1=E_0+\beta\rho$, $E_2=E_0+\alpha\sigma$, and
$c=c_0+m\sigma+\ell\rho$.  For $\alpha,\beta>0$ put
\begin{equation}\label{eq:D-tau}
 \tau=\min(\beta\rho,\alpha\sigma),\qquad
 \zeta=\frac{\ell}{\beta}+\frac{m}{\alpha}.
\end{equation}
Since $\rho\ge\tau/\beta$ and $\sigma\ge\tau/\alpha$, the runner-up
reduction \eqref{eq:D-runner} becomes
\begin{equation}\label{eq:D-pair}
 B-c\le E_0+\tau,\qquad c\ge c_0+\zeta\tau.
\end{equation}
If $\alpha=0$ or $\beta=0$, then $\min(E_1,E_2)=E_0$ and $c\ge c_0$, so
\eqref{eq:D-pair} holds with $\tau=0$ and $\zeta$ is not needed.  The
pair $(B,c)$ is controlled through the balanced one-parameter path
\begin{equation}\label{eq:D-F}
 F(\tau'):=c_0+\zeta\tau'-
 \psi_k\bigl(H_0+(\zeta+1)\tau'\bigr).
\end{equation}

\begin{lemma}[path transfer]\label{lem:D-transfer}
Let $\zeta,\tau\ge0$ and $0\le H_0\le\bar t_k$, and suppose
$F(\tau')\ge0$ for every $0\le\tau'\le\tau$ with path argument
$H_0+(\zeta+1)\tau'\le1$.  If $0\le B\le1$ and
\eqref{eq:D-pair} holds, then $\psi_k(B)\le c$.
\end{lemma}

\begin{proof}
The proof only evaluates $F$ at path arguments that are at most $B\le1$,
so the hypothesis always applies.  If $B\le H_0$, then monotonicity of
$\psi_k$ on the increasing branch
(\eqref{eq:cap-increasing}), $F(0)\ge0$, and \eqref{eq:D-pair} give
$\psi_k(B)\le\psi_k(H_0)\le c_0\le c$.  If $H_0<B\le H_0+(\zeta+1)\tau$,
write $B=H_0+(\zeta+1)\tau'$ with $0<\tau'\le\tau$; then
$\psi_k(B)\le c_0+\zeta\tau'\le c_0+\zeta\tau\le c$.  If
$B>H_0+(\zeta+1)\tau$, read $F(\tau)\ge0$ as
$\delta_k\bigl(H_0+(\zeta+1)\tau\bigr)\ge E_0+\tau$; since $\delta_k$ is
nondecreasing, $\delta_k(B)\ge E_0+\tau\ge B-c$, which is the claim.
\end{proof}

The next two subsections verify the hypotheses: $F(0)\ge0$ and
$H_0\le\bar t_k$ at the boundary, and $F(\tau')\ge0$ for $\tau'>0$ with
path argument in $[0,1]$ --- where $\psi_k$ is concave --- from the
stationary reserve.

\subsection{The boundary \texorpdfstring{$\tau'=0$}{tau' = 0}}

For fixed $t,u$, the boundary margin $c_0-\psi_k(H_0)$ increases with
each actual floor.  Indeed $\partial c_0/\partial m=Q$ and
$\partial H_0/\partial m=-(k-2)\beta$, so differentiation in $m$ gives
\[
 Q+(k-2)\beta\psi_k'(H_0)
 \ge Q-\frac{k-2}{k-1}\beta\ge\frac\beta{k-1},
\]
by $\psi_k'\ge-1/(k-1)$ and \eqref{eq:D-signs}; the other derivative is
symmetric.  Hence set
\begin{equation}\label{eq:D-floor-equality}
 m=\psi_k(t),\qquad \ell=\psi_k(u).
\end{equation}
This substitution can only raise $H_0$, since
$\partial H_0/\partial m=-(k-2)\beta\le0$ and
$\partial H_0/\partial\ell=-(k-2)\alpha\le0$; the cap
$H_0\le\bar t_k$ proved below therefore also covers the actual floors.
With $\pi=t+u-tu$, direct algebra gives
\begin{align}
 c_0&=\psi_k(\pi),\label{eq:D-c0}\\
 E_0&=(1-t)^k\delta_k(u)+\delta_k(t)(1-u)^k
      -\delta_k(t)\delta_k(u).\label{eq:D-E0}
\end{align}
Superadditivity of $\delta_k$ (convexity with $\delta_k(0)=0$) and the
scaling bound $\delta_k(\vartheta x)\ge\vartheta^k\delta_k(x)$ for
$0\le\vartheta\le1$ give, as $(1-u)^k\le1$,
\[
 E_0\le\delta_k(t)+(1-t)^k\delta_k(u)
 \le\delta_k(t)+\delta_k\bigl((1-t)u\bigr)\le\delta_k(\pi),
\]
hence $H_0=\psi_k(\pi)+E_0\le\pi$.  The scaling bound holds because
$\delta_k(x)/x^k$ is nonincreasing:
$k\delta_k(x)-x\delta_k'(x)
=\tfrac k{k-1}\bigl((1-x)^{k-1}-1+(k-1)x\bigr)\ge0$
by Bernoulli's inequality.
Moreover the exact factorization
\begin{equation}\label{eq:D-H-factor}
 \mathcal H_k(\pi)-H_0
 =\bigl((1-t)^k-\delta_k(t)\bigr)
  \bigl((1-u)^k-\delta_k(u)\bigr)\ge0
\end{equation}
gives $H_0\le\mathcal H_k(\pi)$; both factors are nonnegative, since by
\eqref{eq:psi-identities}
$(1-t)^k-\delta_k(t)=1-2t-(k-2)\psi_k(t)\ge1-2t-(k-2)m\ge0$
by \eqref{eq:D-signs}, and symmetrically in $u$.
If $\mathcal H_k(\pi)\ge\pi$, the
inequality is precisely the cap condition for $\pi$, and monotonicity
yields $\psi_k(H_0)\le\psi_k(\pi)=c_0$.  If $\mathcal H_k(\pi)\le\pi$,
apply Lemma~\ref{lem:sharp-family} and the cap for $\mathcal H_k(\pi)$
to reach the same conclusion.  In particular
$H_0\le\min\bigl(\pi,\mathcal H_k(\pi)\bigr)\le\bar t_k$: if
$\pi\le\bar t_k$ this is clear, and otherwise
$\mathcal H_k(\pi)\le\mathcal H_k(\bar t_k)=\bar t_k$, since
$\mathcal H_k$ is decreasing ($\psi_k'\ge-1/(k-1)$) and fixes
$\bar t_k$ by \eqref{eq:cap-root}.  Thus
\begin{equation}\label{eq:D-F0}
 F(0)=c_0-\psi_k(H_0)\ge0.
\end{equation}

\subsection{The stationary reserve}

Since $\psi_k$ is concave, $F'$ is increasing.  At a stationary point with
argument $H_*$,
\[
 \psi_k'(H_*)=\frac \zeta{\zeta+1},\qquad
 \delta_k'(H_*)=\frac1{\zeta+1}.
\]
The explicit derivative formula \eqref{eq:psi-derivatives} then gives
\begin{equation}\label{eq:D-stationary-psi}
 \psi_k(H_*)=\frac{1-H_*}{k(\zeta+1)}\le\frac1{k(\zeta+1)}.
\end{equation}
If $H_*\le H_0$, $F$ is increasing and \eqref{eq:D-F0} suffices.
Otherwise convexity and \eqref{eq:D-stationary-psi} give
\[
 F(\tau')\ge c_0-\psi_k(H_*)
 \ge c_0-\frac1{k(\zeta+1)}.
\]
Nonnegativity of $F$ at every path argument in $[0,1]$ --- all that
Lemma~\ref{lem:D-transfer} requires --- therefore follows from
\begin{equation}\label{eq:D-reserve}
 k(\zeta+1)c_0\ge1.
\end{equation}

\begin{lemma}[D-reserve]\label{lem:D-reserve}
Under \eqref{eq:D-floor-equality}, the reserve
\eqref{eq:D-reserve} holds for every integer $k\ge4$.
\end{lemma}

\begin{proof}
Assume without loss of generality $0<u\le t\le\bar t_k$; for $u=0$ the
reserve is never invoked, since then $\beta=0$ and
Lemma~\ref{lem:D-transfer} applies with $\tau=0$.
For $k\ge5$, monotonicity of $\delta_k(x)/x$ gives
\[
 \frac1{k(\zeta+1)}\le
 \frac1k\min\left(\frac{\delta_k(t)}t,
                    \frac{\delta_k(u)}u\right)
 =\frac{\delta_k(u)}{ku}.
\]
Thus it is enough to prove
\begin{equation}\label{eq:D-key-qge5}
 ku\,\psi_k(t+u-tu)\ge\delta_k(u).
\end{equation}
For fixed $u$, the left side minus the right side is concave in $t$, so its
minimum occurs at $t=u$ or $t=\bar t_k$.  Appendix~\ref{app:D-reserve}
proves both endpoint inequalities for every $k\ge5$.

For $k=4$, \eqref{eq:D-key-qge5} is slightly false near the cap corner, but
the full reserve remains true.  The direct estimate in
Appendix~\ref{app:D-reserve} uses
\[
 \frac13\le\bar t_4\le\frac9{26},\qquad
 \frac{\psi_4(t)}{\delta_4(t)}\ge\frac45,\qquad
 \delta_4(u)\le2u^2
\]
and closes the two endpoint rectangles exactly.
\end{proof}

\begin{remark}[reserve at actual floors]\label{rem:D-reserve-floors}
The reserve \eqref{eq:D-reserve} persists for all floors
$m\ge\psi_k(t)$, $\ell\ge\psi_k(u)$.  Writing $c_0^{\mathrm{eq}}$,
$\zeta^{\mathrm{eq}}$ for the values at floor equality
\eqref{eq:D-floor-equality}, exact expansion gives
\[
 c_0-c_0^{\mathrm{eq}}
 =(1-u)^k\bigl(m-\psi_k(t)\bigr)+P\bigl(\ell-\psi_k(u)\bigr)\ge0,
\]
using $P\ge0$ from \eqref{eq:D-signs}, while
$\zeta\ge\zeta^{\mathrm{eq}}$ termwise, because $x\mapsto x/(w-x)$
increases on $[0,w)$.  Since $c_0^{\mathrm{eq}}=\psi_k(\pi)\ge0$, the
left side of \eqref{eq:D-reserve} only grows.
\end{remark}

\begin{theorem}[D-core]\label{thm:D-core}
For two D-type members in the branch $B>z$, one has
$c\ge\psi_k(B)$.
\end{theorem}

\begin{proof}
The pair $(B,c)$ satisfies \eqref{eq:D-pair}, with $\tau=0$ when
$\alpha\beta=0$, and $0\le B\le\bar t_k\le1$ by
Lemma~\ref{lem:mul-pairing-cap}.  The boundary subsection established
$F(0)\ge0$ and $H_0\le\bar t_k$; for $\tau'>0$ with path argument in
$[0,1]$, $F(\tau')\ge0$ by the
stationary argument together with Lemma~\ref{lem:D-reserve} and
Remark~\ref{rem:D-reserve-floors}.  Lemma~\ref{lem:D-transfer} gives
$\psi_k(B)\le c$.
\end{proof}

\section{Multiplication closure and the stability theorem}
\label{sec:assembly}

\begin{theorem}[absorbing-zero quasigroup gate]\label{thm:multiplication}
For every absorbing-zero quasigroup operation on a set of order $k\ge4$,
convolution preserves $\body_k$.  In particular, multiplicative convolution
preserves $\body_k$ on every finite field of order $k$.
\end{theorem}

\begin{proof}
Let $\mu,\nu\in\body_k$ and use the notation of
Section~\ref{sec:mul-reduction}.

If one input is K-type, product zero is largest.  Theorem~\ref{thm:K-core}
and Lemma~\ref{lem:mul-rearrangement} give
\[
 \min_{g\ne0}w_g\ge A\ge\psi_k(C)
 \ge\psi_k\left(\max_{g\ne0}w_g\right),
\]
where the last comparison is on the cap:
$\max_{g\ne0}w_g\le C\le\bar t_k$ by Lemma~\ref{lem:mul-rearrangement}
and Theorem~\ref{thm:K-core}.  This is membership in $\body_k$.

Suppose both inputs are D-type.  Equations \eqref{eq:DD-baseline} and
\eqref{eq:z-ge-c} give $z,A\ge c\ge\psi_k(z)$.

If zero is not largest and $B\le z$, the global second atom equals $z$:
the largest atom is nonzero, and every other nonzero atom is at most
$B\le z$ by Lemma~\ref{lem:mul-rearrangement}.  The global minimum is at
least $\min(z,A)\ge\psi_k(z)$ by \eqref{eq:DD-baseline} and
$\psi_k(z)\le z$, so this is membership with no monotonicity needed.
If zero is not largest and $B>z$, Theorem~\ref{thm:D-core}
and the runner-up bound give
\[
 \min(z,A)\ge c\ge\psi_k(B)
 \ge\psi_k\bigl(p_2^\down(\mu\otimes\nu)\bigr),
\]
where the final comparison is on the cap:
$p_2^\down(\mu\otimes\nu)\le\max(z,B)=B\le\bar t_k$ by
Lemma~\ref{lem:mul-rearrangement} and Lemma~\ref{lem:mul-pairing-cap}.

It remains to suppose zero is largest.  Every nonzero atom is at least
$\psi_k(z)$ by \eqref{eq:DD-baseline}.  Their largest atom $H$ satisfies
\[
 H\le z,\qquad H\le1-z-(k-2)\psi_k(z)=\mathcal H_k(z).
\]
If $\mathcal H_k(z)\ge z$, then $z$ is in the cap and
$\psi_k(H)\le\psi_k(z)$.  If $\mathcal H_k(z)\le z$, the same conclusion
follows from Lemma~\ref{lem:sharp-family}, together with the cap for
$\mathcal H_k(z)$.  Again the product belongs to $\body_k$.
\end{proof}

\begin{proof}[Proof of Theorem~\ref{thm:quasigroup}]
The two assertions are Theorems~\ref{thm:addition}
and~\ref{thm:multiplication}.  Applying them independently also proves the
combined statement, since no compatibility between the operations enters
either proof.
\end{proof}

\begin{proof}[Proof of Theorem~\ref{thm:body}]
Every group operation is a quasigroup operation.  Finite-field
multiplication has an absorbing zero and restricts to a group, hence a
quasigroup, on the nonzero elements.  Theorem~\ref{thm:quasigroup} therefore
specializes to the two field gates.
\end{proof}

\begin{remark}[the inequality alone is not an invariant]
\label{rem:obstruction}
The full region defined by $p_1^\down-p_k^\down\le(1-p_2^\down)^k$ need
not be closed under multiplicative convolution.  Already over $\F_5$,
the two rational laws with atom numerators
\[
 (10910,63824,9750,7809,7707),\qquad
 (10329,9177,64827,6853,8814)
\]
at the field elements $(0,1,2,3,4)$ and common denominator $100000$
each satisfy the inequality, while their product does not.
Theorem~\ref{thm:multiplication} shows that the more structured
condition of membership in $\body_k$ is preserved.
\end{remark}

\clearpage
\appendix

\section{Formalization in Lean}\label{sec:lean}

The complete development --- both field gates, their quasigroup
extensions, the read-once corollary, and the geometric results --- is
formalized in Lean~4 on top of Mathlib.  The source is available in the
companion repository
\href{https://github.com/Iluvmagick/rronce}{\texttt{Iluvmagick/rronce}}.
The Lean identifiers retain the letter $q$ for the order where the paper
writes $k$.  The geometric statements use the index $m:=q-2$ and the
standard coordinate chart of the probability simplex: a law on $q=m+2$
atoms is a vector $x\in\R^{m+1}$ together with the omitted atom
$1-\sum_ix_i$, and the normalized fraction is
$\operatorname{vol}_{m+1}(\body_q)/\operatorname{vol}_{m+1}(\Delta_{m+1})$,
exactly the normalization in \eqref{eq:B-volume-lower}.

Twenty endpoints are audited.  The principal ones are
\begin{center}
\texttt{Rronce.Bq.add\_gate\_preserves\_Bq},\qquad
\texttt{Rronce.Bq.mul\_gate\_preserves\_Bq},\\
\texttt{Rronce.Bq.quasigroup\_gates\_preserve\_Bq},\qquad
\texttt{Rronce.General.main\_membership\_qge4},\\
\texttt{Rronce.Bq.volume\_fraction\_lower\_bound},\qquad
\texttt{Rronce.Bq.exact\_volume\_formula},\\
\texttt{Rronce.Bq.exact\_volume\_root\_asymptotics}.
\end{center}
The first two are Theorem~\ref{thm:body} and the third is
Theorem~\ref{thm:quasigroup}; \texttt{main\_membership\_qge4} is
Corollary~\ref{cor:readonce}, stated on the exact rational output
probabilities; the last three are the volume bound
\eqref{eq:B-volume-lower} and the exact-volume results of
Appendix~\ref{app:exact-volume}.  The remaining endpoints record the
separate quasigroup forms, the derived inequality
$p_1^\down-p_k^\down\le(1-p_2^\down)^k$ over every finite field of
order at least two (orders two and three proved separately), the
nonvacuity statements behind Proposition~\ref{prop:B-interior}, and the
companion measurability, positivity, and logarithmic-rate facts.

Each principal endpoint is stated in a proof-free statement module,
separate from its implementation, and \texttt{AxiomCheck.lean} restates
all twenty endpoints --- against those statement modules where they
exist, and as fully spelled-out statements otherwise --- so that the
file typechecks only if the proved theorems are exactly the stated
ones.  The scalar endpoint facts used in
Appendices~\ref{app:rank}--\ref{app:D-reserve} name their certifying
modules where they are invoked.  The theorem chain contains no
\texttt{sorry}, \texttt{admit}, \texttt{native\_decide}, unsafe
declaration, or custom axiom: \texttt{\#print axioms} reports exactly
\[
 \texttt{propext},\qquad \texttt{Classical.choice},\qquad
 \texttt{Quot.sound}
\]
for every endpoint.  The verification commands are
\begin{lstlisting}[language=bash]
lake build
lake env lean AxiomCheck.lean
\end{lstlisting}
The statements of the quasigroup and geometric endpoints are reproduced
verbatim in Appendix~\ref{app:lean-statement}.
For background on Lean and Mathlib, see \cite{demoura2021lean4,mathlib2020}.

\section{The K-rank endpoint certificates}\label{app:rank}

We complete the proof of Lemma~\ref{lem:K-rank}; symbols introduced in
this appendix are local to it.  Let
\[
 y=\sqrt{k/(2\lambda)}>1,\qquad \varepsilon=1/\lambda,\qquad
 u=1/\varphi_0,\qquad \mathcal R=\xi_0/\lambda.
\]
The definitions in the proof of the lemma give
\begin{equation}\label{eq:rank-transform}
 u^2=2(y+1)+\frac{1-\varepsilon}{y-1},\qquad
 \mathcal R=1+\varepsilon+(y-1)u.
\end{equation}
After division by $\lambda$, \eqref{eq:rank-boundary-target} becomes
\begin{equation}\label{eq:rank-E-target}
 \mathcal R\left(1-\frac{\varepsilon}{\mathcal R}\right)^{2y^2/\varepsilon}
 <1-\frac1u.
\end{equation}
The domain is $0<\varepsilon\le\min(1,y^2/2)$ and $u\ge2$.

For fixed $y$, the left side of \eqref{eq:rank-E-target} increases in
$\varepsilon$ and the right side decreases.  Indeed, put
$z=\varepsilon/\mathcal R$.  From
\eqref{eq:rank-transform},
$\mathcal R'=1-1/(2u)\ge3/4$ and $0<z\le1/2$.  Differentiation gives
\[
 \frac{\mathrm d}{\mathrm d \varepsilon}\log\left[\mathcal R
   \left(1-\frac \varepsilon{\mathcal R}\right)^{2y^2/\varepsilon}\right]
 =\frac{\mathcal R'}{\mathcal R}
 -\frac{2y^2}{\varepsilon^2}
 \left(\log(1-z)+\frac{z(1-z\mathcal R')}{1-z}\right)>0,
\]
where $-\log(1-z)\ge z+z^2/2$ and
$(1-\mathcal R')/(1-z)\le1/2$.  Also $u$ decreases with $\varepsilon$.
It is therefore enough to take
\begin{equation}\label{eq:rank-Emax}
 \varepsilon=\min(1,y^2/2).
\end{equation}

If $y\ge\sqrt2$, then $\varepsilon=1$, and \eqref{eq:rank-transform} gives
$u=\sqrt{2(y+1)}$ together with
\[
 k'=2y^2=\frac{(u^2-2)^2}{2},\qquad
 \mathcal R=\frac{u^3-4u+4}{2}.
\]
The target is equivalent to
\begin{equation}\label{eq:rank-U-target}
 \left(1+\frac1{\mathcal R-1}\right)^{k'}
 >\frac{\mathcal R u}{u-1}.
\end{equation}
For $k'\ge5$, the degree-five Taylor lower bound has difference
\[
 \frac{(u-2)(u^4-4u^2+2)P_{16}(u)}
 {120(u-1)(u^3-4u+2)^5},
\]
where $P_{16}(2+w)$ has coefficients
\[
\begin{aligned}
 &1,38,611,5506,30832,111930,266270,412516,423588,\\
 &360156,409976,493392,385232,164016,34368,3776,512.
\end{aligned}
\]
All are positive.  For $4\le k'\le5$, the degree-four Taylor bound has
difference
\[
 \frac{(u-2)^2(u^4-4u^2+2)P_{11}(u)}
 {24(u-1)(u^3-4u+2)^4}.
\]
Here $2<u<23/10$ and, with $\zeta=(10/3)(u-2)\in(0,1)$,
\begin{align*}
 P_{11}(u)={}&672(1-\zeta)^{11}+\frac{41328}{5}\zeta(1-\zeta)^{10}
 +\frac{229434}{5}\zeta^2(1-\zeta)^9+151677\zeta^3(1-\zeta)^8\\
 &+\frac{207146862}{625}\zeta^4(1-\zeta)^7
 +\frac{3138874749}{6250}\zeta^5(1-\zeta)^6
 +\frac{10756413261}{20000}\zeta^6(1-\zeta)^5\\
 &+\frac{2031792649713}{5000000}\zeta^7(1-\zeta)^4
 +\frac{2647690169319}{12500000}\zeta^8(1-\zeta)^3\\
 &+\frac{18078264722949}{250000000}\zeta^9(1-\zeta)^2
 +\frac{144849864078153}{10000000000}\zeta^{10}(1-\zeta)\\
 &+\frac{128049353071077}{100000000000}\zeta^{11}>0.
\end{align*}
Thus \eqref{eq:rank-U-target} holds.

If $1<y\le\sqrt2$, then $\varepsilon=y^2/2$ and the transformed exponent
is four.  Put $a=y-1$ and $Z=1+au$.  Then
\[
 0<a<5/12,\qquad
 u^2=\frac{3a^2+6a+1}{2a},\qquad
 \mathcal R=Z+\frac{(a+1)^2}{2}.
\]
The desired numerator is
\[
 N=(u-1)\mathcal R^3-uZ^4.
\]
Reduction modulo the quadratic relation gives
\[
 8N=-(a+1)^2(uC(a)+D(a)),
\]
where
\begin{align*}
 C(a)&=17a^4+20a^3-30a^2+44a-19,\\
 D(a)&=-8a^4+40a^3+72a^2-56a+16.
\end{align*}
On $0<a<5/12$, one has $-C(a)>0$, $D(a)>0$, and
\[
 u^2C(a)^2-D(a)^2=\frac{(a+1)^4P_6(a)}{2a}>0,
\]
where
\[
 P_6(a)=867a^6+178a^5-2125a^4+916a^3
 +1393a^2-1462a+361.
\]
All three signs admit positive Bernstein representations after the scaling
$a=(5/12)\zeta$, $0<\zeta<1$; their exact positive coefficients are given in
\texttt{KCoreRankEndpoints.lean}.  Hence $u(-C)>D$, so $N>0$.
This completes both endpoints in \eqref{eq:rank-Emax} and proves
Lemma~\ref{lem:K-rank}.

\section{The sharp mixed inequality}\label{app:mixed-sharp}

We prove Lemma~\ref{lem:mixed-sharp} on the direct cap
$0\le t,p\le T:=\bar t_k$; symbols introduced in this appendix are local
to it.  Both sides of \eqref{eq:mixed-sharp} vanish at $p=0$, so assume
$p>0$.  Put
\[
 n=k-1,\qquad N=k-2,\qquad b=1-t,\qquad a=1-T.
\]
Convexity of $\delta_k$ reduces \eqref{eq:mixed-sharp} to
\begin{equation}\label{eq:mixed-tangent}
 \delta_k(bp)-b^k\delta_k(p)
 \ge N\psi_k(t)\delta_k(p)\delta_k'(bp).
\end{equation}
The tangent point is nonnegative because the cap gives
$N\psi_k(t)\le1-2t\le b$.

The integral formula
\begin{equation}\label{eq:delta-integral}
 \delta_k(x)=k\int_0^x(x-r)(1-r)^N\,\mathrm d r
\end{equation}
supplies a scaling reserve.  Let
$I_j(p)=\int_0^p(p-r)(1-r)^j\,\mathrm d r$; since $p>0$, every
$I_j(p)>0$ and $\delta_k(p)=kI_N(p)>0$.  Expanding
$(1-br)^N$ gives
\begin{equation}\label{eq:mixed-integral-expand}
 \frac{\delta_k(bp)-b^k\delta_k(p)}{\delta_k(p)}
 =b^2\sum_{j=1}^N\binom Nj t^jb^{N-j}\frac{I_{N-j}(p)}{I_N(p)}.
\end{equation}
For the probability measure proportional to
$(p-r)(1-r)^N\,\mathrm d r$,
\begin{equation}\label{eq:mixed-moment}
 \E[r]\ge\frac{p}{3+Np}.
\end{equation}
To see this, scale $r=px$, put $W(x)=(1-px)^N$, and integrate the derivative
of $x(1-x)^2W(x)$.  If
$I_1=\int_0^1(1-x)W(x)\,\mathrm d x$ and
$I_2=\int_0^1x(1-x)W(x)\,\mathrm d x$, the result is
$I_1-3I_2\le NpI_2$.

Bernoulli's inequality now yields
\[
 \frac{I_{N-j}(p)}{I_N(p)}
 =\E[(1-r)^{-j}]
 \ge1+\frac{jp}{3+Np}.
\]
Substitution into \eqref{eq:mixed-integral-expand} reduces
\eqref{eq:mixed-tangent} to
\begin{equation}\label{eq:mixed-elementary}
 S_n(b)(1-bp)^n\ge1-bL(p),
\end{equation}
where
\[
 S_n(b)=1+b+\cdots+b^{n-1},\qquad
 L(p)=\frac{Np}{3+Np}.
\]
Equivalently,
\begin{equation}\label{eq:Phi}
 \Phi(b,p):=S_n(b)(1-bp)^n
 \frac{3+Np}{3+N(1-b)p}\ge1.
\end{equation}
One has
\[
 \partial_p\log\Phi
 =b\left(-\frac n{1-bp}
 +\frac{3N}{(3+Np)(3+N(1-b)p)}\right)<0,
\]
so it suffices to set $p=T$.

For $n\ge5$, $\log S_n$ is convex on $[0,1]$.  Indeed the sign of its
second derivative is the sign of
\[
 Q_n(b)=S_n(b)^2-nb^{n-2}(n-1+b^n)=(1-b)^2C_n(b),
\]
and the coefficients of $C_n$ have one sign change, from positive at lower
powers to negative at higher powers, while
$C_n(1)=n^2(n-1)(n-5)/12\ge0$.  Therefore
\begin{equation}\label{eq:Phi-b-derivative}
 \partial_b\log\Phi(b,T)
 =\frac{S_n'(b)}{S_n(b)}-\frac{nT}{1-bT}
 +\frac{NT}{3+N(1-b)T}\ge0
\end{equation}
once
\begin{equation}\label{eq:mixed-cap-s}
 1-3s-2(n-1)s^2\ge0,\qquad s=(1-T)^n.
\end{equation}
Indeed, on the cap the identity $(1-T)\bigl(n+1+(n-1)s\bigr)=n$ turns
\eqref{eq:mixed-cap-s} into the endpoint slope bound
$S_n'(1-T)/S_n(1-T)\ge nT/(1-T)$; convexity of $\log S_n$ propagates it
to every $b\in[1-T,1]$, where $nT/(1-bT)\le nT/(1-T)$ and the last term
of \eqref{eq:Phi-b-derivative} is nonnegative.
For $n\ge10$, the bound $s\le1/(2\sqrt n)$ --- proved as
\eqref{eq:D-s-upper} in Appendix~\ref{app:D-reserve}, whose $s$ is the
same quantity --- proves
\eqref{eq:mixed-cap-s}.  For $5\le n\le9$ it follows by exact rational
substitution using the cap intervals in Table~\ref{tab:cap-intervals}.
The case $n=4$ is direct: $S_4'/S_4\ge3/2$ on $[1-T,1]$, while the
remaining two terms of \eqref{eq:Phi-b-derivative} combine to less than
$3/2$ in magnitude there:
$nT/(1-bT)-NT/(3+N(1-b)T)<3/2$.  Hence for $k\ge5$ the minimum
is at $b=a=1-T$.  For $k=4$, the derivative has the opposite sign and the
minimum is at $b=1$.

\begin{table}[ht]
\centering
\begin{tabular}{c@{\qquad}c@{\qquad}c}
\toprule
$n=k-1$ & lower bound for $T$ & upper bound for $T$\\
\midrule
3 & $343/1000$ & $344/1000$\\
4 & $300/1000$ & $301/1000$\\
5 & $268/1000$ & $269/1000$\\
6 & $243/1000$ & $244/1000$\\
7 & $223/1000$ & $224/1000$\\
8 & $207/1000$ & $208/1000$\\
9 & $193/1000$ & $194/1000$\\
\bottomrule
\end{tabular}
\caption{Exact rational brackets for the finitely many small cap roots.}
\label{tab:cap-intervals}
\end{table}

At $k=4$, the endpoint is
$\Phi(1,T)=(1-T)^3(3+2T)>1$; Table~\ref{tab:cap-intervals} gives the
positive lower margin $9902181/244140625$.

For $k\ge5$, put $s=(1-T)^n$, $E=1+(n-1)s$, and $q_0=n+E$.  The cap equation
gives $1-T=n/q_0$, $T=E/q_0$.  Bernoulli's inequality reduces
$\Phi(1-T,T)\ge1$ to
\begin{equation}\label{eq:mixed-LR}
 L_n(s):=\frac{s(1-s)(q_0+E^2)}E
 \ge
 R_n(s):=\frac{3q_0^2+(n-1)E^2}{q_0(3q_0+(n-1)E)}.
\end{equation}
On $s\le1/(2\sqrt n)$, $L_n$ increases and $R_n$ decreases; for
$n\ge5$ the actual $s$ lies in this range by \eqref{eq:D-s-upper}, and
only the $n\ge10$ argument below uses it.  For
$n\ge10$ one also has $s\ge4/(3n)$: if $s<4/(3n)$, then
$E<7/3-4/(3n)$, and the cap equation $s=(1+E/n)^{-n}$ would give
$s>4/(3n)$, because
$\bigl(1+(7/3-4/(3n))/n\bigr)^n<3n/4$ --- by
$e^{7/3}<27\le3n/4$ for $n\ge36$ and a direct check for
$10\le n\le35$.  Substitution of $4/(3n)$ in
\eqref{eq:mixed-LR} leaves a positive denominator and numerator
\begin{align*}
 P(n)={}&81n^8+18585n^7+28530n^6-221464n^5+124544n^4\\
 &+152192n^3-180224n^2+65536n-8192>0.
\end{align*}
For $4\le n\le9$, Table~\ref{tab:cap-intervals} gives
\eqref{eq:mixed-LR} by exact rational substitution.  All these estimates
are also kernel-checked: the finite rational substitutions in
\texttt{MixedSharp/Proof.lean}, the uniform $n\ge10$ endpoint bounds
(including $s\ge4/(3n)$, as \texttt{test\_cap\_s\_lower}) in
\texttt{MixedSharp/EndpointLarge.lean}, and the $k=4$ slope reversal in
\texttt{MixedSharp/Q4Slope.lean}.  This proves \eqref{eq:Phi}, then
\eqref{eq:mixed-tangent}, and finally Lemma~\ref{lem:mixed-sharp} on the
direct cap.

\section{The D-reserve endpoints}\label{app:D-reserve}

We prove the endpoint estimates used in Lemma~\ref{lem:D-reserve};
symbols introduced in this appendix are local to it.  Let
$T=\bar t_k$, $a=1-T$, $b=1-u$, and $n=k-1$.  The cap equation is
\begin{equation}\label{eq:D-cap-complement}
 ka+(k-2)a^k=n.
\end{equation}

\subsection{The estimate for \texorpdfstring{$k\ge5$}{k >= 5}}

For fixed $u$, concavity in $t$ reduces
\eqref{eq:D-key-qge5} to $t=u$ and $t=T$.  After multiplication by $n$,
the two endpoint margins are
\begin{align}
 M_k(b)&=k(1-b)b^2(1-b^{2n})-(n-kb+b^k),\label{eq:D-Aq}\\
 N_k(b)&=k(1-b)ab(1-(ab)^n)-(n-kb+b^k).\label{eq:D-Bq}
\end{align}
We show both nonnegative on $a\le b\le1$.

First take the cap corner $b=a$ and set $s=a^n$.  Equation
\eqref{eq:D-cap-complement} becomes
\begin{equation}\label{eq:D-cap-s}
 a\bigl((n+1)+(n-1)s\bigr)=n.
\end{equation}
The corner inequality is equivalent to $E_n(s)\ge0$, where
\begin{equation}\label{eq:D-En}
 E_n(s)=(n+1)-2(n+1)s-(2n^2+n-1)s^2-2n(n-1)s^3.
\end{equation}
This polynomial decreases for $s\ge0$.  For $n=4$, evaluation of the cap
functional at $a=7/10$ gives $a<7/10$, hence
\[
 E_4(s)>E_4(2401/10000)
 =\frac{31141094647}{125000000000}>0.
\]
For $n\ge5$ one has
\begin{equation}\label{eq:D-s-upper}
 s\le\frac1{2\sqrt n}.
\end{equation}
To prove it, substitute the value $1/(2\sqrt n)$ for $s$ in
\eqref{eq:D-cap-s}.  The required comparison follows from
\[
 \log(2\sqrt n)\le E-\frac{E^2}{2n},\qquad
 E=1+\frac{n-1}{2\sqrt n};
\]
the difference increases in $\sqrt n$ and is positive at $n=5$.
At $s_0:=1/(2\sqrt n)$, \eqref{eq:D-En} becomes
\[
 \frac{2r^4-5r^3+3r^2-3r+1}{4r^2},\qquad r=\sqrt n,
\]
which is positive at $r=\sqrt5$ and increasing thereafter; since $E_n$
decreases and $s\le s_0$, this proves
$M_k(a)\ge0$.

For general $b$, exact division yields
\begin{equation}\label{eq:D-A-factor}
 M_k(b)=(1-b)^2R_k(b),
\end{equation}
where
\[
 R_k(b)=-n-(n-1)b+\sum_{j=2}^{n-1}(j+1)b^j
 +k\sum_{j=n}^{2n+1}b^j.
\]
Every term of $R_k''$ is positive, while $R_k(0),R_k'(0)<0$.  Thus
$R_k$ has at most one positive zero.  Since $R_k(a)\ge0$, it is
nonnegative throughout $[a,1]$, proving \eqref{eq:D-Aq}.

Similarly,
\begin{equation}\label{eq:D-B-factor}
 N_k(b)=(1-b)S_k(b),\qquad
 S_k(b)=-n+(ka+1)b+\sum_{j=2}^n b^j-ka^kb^k.
\end{equation}
At $b=a$, equations \eqref{eq:D-Aq} and \eqref{eq:D-Bq} agree, so
$S_k(a)\ge0$.  Also
\[
 S_k'(b)\ge ka+1+
 \left(\frac{n(n+1)}2-1-k^2a^k\right)b^n>0;
\]
the last coefficient is positive for $k\ge5$ because the preceding cap
bounds give $a^k<1/4$.  Hence $S_k$ increases and $N_k(b)\ge0$.
This proves \eqref{eq:D-key-qge5}.

\subsection{The exceptional estimate at \texorpdfstring{$k=4$}{k = 4}}

For $k=4$,
\[
 \psi_4(x)=x-2x^2+\frac43x^3-\frac13x^4,\qquad
 \delta_4(x)=x^2\left(2-\frac43x+\frac13x^2\right).
\]
The cap root satisfies
\begin{equation}\label{eq:D-q4-cap}
 \frac13\le T\le\frac9{26},
\end{equation}
by exact evaluation of $2x+2\psi_4(x)-1$ at the two endpoints.  Concavity
of $\psi_4$ and the cap identity give
\begin{equation}\label{eq:D-q4-ratio}
 \frac{\psi_4(t)}{\delta_4(t)}\ge\frac45,\qquad0<t\le T.
\end{equation}
Consequently
\[
 \zeta+1=\frac{u}{\delta_4(u)}+
       \frac{\psi_4(t)}{\delta_4(t)}
 \ge\frac{u}{\delta_4(u)}+\frac45.
\]
Since $\delta_4(u)\le2u^2$,
\begin{equation}\label{eq:D-q4-stationary}
 \frac1{4(\zeta+1)}
 \le\frac{u}{2+(16/5)u}.
\end{equation}
Put $\pi=t+u-tu$.  From \eqref{eq:D-q4-cap}, $\pi<3/5$, and the
factorization
\[
 \psi_4(\pi)=\frac{\pi(1-\pi)(\pi^2-3\pi+3)}3
 \ge\frac{\pi(1-\pi)}2
\]
reduces the reserve to
\begin{equation}\label{eq:D-q4-last}
 \pi(1-\pi)(5+8u)\ge5u.
\end{equation}
For fixed $u$ the left side is concave in $t$, so check $t=u$ and $t=T$.
At $t=u$, after cancelling $u$, the margin is
\[
 \frac{5-9u-20u^2+27u^3-8u^4}{5}
 \ge\frac{112687}{1142440}>0
\]
on $[0,9/26]$.  At $t=T$, set
\[
 J(T,u)=(1-T)(1-u)(T+(1-T)u)(5+8u)-5u.
\]
It is concave in $T$, so enlarge to the rectangle
$1/3\le T\le9/26$, $0\le u\le9/26$ and check the two $T$-ends.
Each resulting cubic is concave in $u$, leaving the four corner values
\[
 \frac{10}{9},\qquad\frac{7103}{39546},\qquad
 \frac{765}{676},\qquad\frac{1014183}{5940688},
\]
all positive.  This proves \eqref{eq:D-q4-last}, then
\eqref{eq:D-reserve}, at $k=4$.

\section{Exact relative volume and exponential scale}
\label{app:exact-volume}

This appendix supplies the volume arguments deferred from
Section~\ref{sec:invariant}; symbols introduced in this appendix are local
to it.
Write $[x]_+:=\max\{x,0\}$, and let
\[
 \mathfrak v_k:=
 \frac{\operatorname{vol}_{k-1}(\body_k)}
      {\operatorname{vol}_{k-1}(\Delta_{k-1})}
\]
be the fraction of the probability simplex occupied by $\body_k$.  The
ratio is independent of the affine coordinate chart used to define relative
volume.

\begin{proposition}[exact volume and sharp exponential rate]
\label{prop:B-volume-exact}
Let $k\ge4$ and put $N:=k-2$.  Then
\begin{equation}\label{eq:B-volume-exact}
 \mathfrak v_k
 =k(k-1)^2\int_0^{1/2}\sum_{j=0}^{N}(-1)^j\binom Nj
 \bigl[1-(j+2)t-(N-j)\psi_k(t)\bigr]_+^N\,\mathrm d t.
\end{equation}
There is a unique $c_*>0$ satisfying
\begin{equation}\label{eq:B-volume-saddle}
 \frac{1-e^{-c_*}}{c_*}
 =1-\frac{c_*}{e^{c_*}-1}.
\end{equation}
With
\begin{equation}\label{eq:B-volume-rate}
 \Lambda_*:=\log c_*-c_*-1+\frac{1-e^{-c_*}}{c_*},
 \qquad \beta_*:=e^{\Lambda_*},
\end{equation}
one has
\begin{equation}\label{eq:B-volume-asymptotic}
 \lim_{k\to\infty}\frac{\log\mathfrak v_k}{k}=\Lambda_*,
 \qquad
 \lim_{k\to\infty}\mathfrak v_k^{1/k}=\beta_*.
\end{equation}
Equivalently, $\mathfrak v_k=\exp((\Lambda_*+o(1))k)$.  Numerically, by
external evaluation of the saddle equations (the Lean development
certifies the exact saddle-defined limits, not these decimal digits),
\[
 \begin{aligned}
 c_*&=1.392614330471526\ldots,\\
 \Lambda_*&=-1.521745139795282\ldots,\\
 \beta_*&=0.218330536985878\ldots.
 \end{aligned}
\]
\end{proposition}

\begin{proof}
The set on which two atoms tie is contained in a finite union of affine
hyperplanes and hence has relative volume zero.  Off this null set, choose
the labels of the largest and second-largest atoms; there are $k(k-1)$
ordered choices, all meeting $\body_k$ in the same volume by permutation
symmetry of $\body_k$.  Fix one such choice, write the
second-largest atom as $t$, and denote the remaining $N$ atoms by
$\xi_1,\ldots,\xi_N$.  The omitted largest atom is
$1-t-\sum_i\xi_i$.  Thus, up to null boundary faces, the part of this
order chamber lying in $\body_k$ is described by
\[
 0\le t\le\frac12,\qquad
 \psi_k(t)\le\xi_i\le t,\qquad
 \sum_{i=1}^N\xi_i\le1-2t.
\]
After the translation $\eta_i=\xi_i-\psi_k(t)$, its fiber at $t$ is
\[
 0\le\eta_i\le t-\psi_k(t),\qquad
 \sum_{i=1}^N\eta_i\le1-2t-N\psi_k(t).
\]
For $w\ge0$, inclusion--exclusion in the coordinate constraints gives the
clipped-cube identity
\begin{equation}\label{eq:clipped-cube-volume}
 \operatorname{vol}_N\left\{x\in[0,w]^N:\sum_i x_i\le L\right\}
 =\frac1{N!}\sum_{j=0}^N(-1)^j\binom Nj[L-jw]_+^N.
\end{equation}
Indeed, after prescribing a set of $j$ coordinates to exceed $w$ and
translating those coordinates by $w$, the contribution is the simplex of
volume $[L-jw]_+^N/N!$; summing over the prescribed set gives
\eqref{eq:clipped-cube-volume}.  Substitution of
$w=t-\psi_k(t)$ and $L=1-2t-N\psi_k(t)$ now gives the integrand in
\eqref{eq:B-volume-exact}.  Finally, the full simplex has coordinate-chart
volume $1/(k-1)!$, so normalization multiplies the single-chamber integral
by
\[
 k(k-1)\frac{(k-1)!}{N!}=k(k-1)^2.
\]
This proves the exact formula.

We turn to the exponential rate.  Set
\[
 s_k(c):=\left(1-\frac ck\right)^k,\qquad
 w_k(c):=c-1+s_k(c),\qquad
 \alpha_k(c):=s_k(c)-\frac{c-1}{N},
\]
and, for $w\ge0$, let
\[
 W_N(w,\alpha):=
 \operatorname{vol}_N\left\{x\in[0,w]^N:\sum_i x_i\le N\alpha\right\}.
\]
The substitution $t=c/k$ followed by dilation of the $N$ fiber coordinates
by $k-1$ rewrites the preceding chamber decomposition as
\begin{equation}\label{eq:B-volume-scaled}
 \mathfrak v_k=P_k\int_0^{k/2}W_N(w_k(c),\alpha_k(c))\,\mathrm d c,
 \qquad
 P_k:=\frac{(k-1)(k-1)!}{(k-1)^N}.
\end{equation}
Stirling's formula gives
\begin{equation}\label{eq:B-volume-prefactor}
 \lim_{k\to\infty}\frac1k\log P_k=-1.
\end{equation}

Put
\[
 w(c):=c-1+e^{-c},\qquad \alpha(c):=e^{-c},
\]
and, for $\lambda>0$,
\[
 Z_\lambda(w):=\int_0^w e^{-\lambda x}\,\mathrm d x
 =\frac{1-e^{-\lambda w}}{\lambda},\qquad
 F_\lambda(c):=\lambda\alpha(c)+\log Z_\lambda(w(c)).
\]
The function
\[
 g(c):=\frac{1-e^{-c}}c+\frac c{e^c-1}-1
\]
is strictly decreasing on $(0,\infty)$: the derivatives of its first two
summands have numerators $(c+1)e^{-c}-1<0$ and
$e^c(1-c)-1<0$, respectively.  Since $g(0+)=1$ and
$g(c)\to-1$ as $c\to\infty$, equation~\eqref{eq:B-volume-saddle} has a
unique positive solution.

Let
\[
 w_*:=w(c_*),\qquad \lambda_*:=\frac{c_*}{w_*}.
\]
A rearrangement of the saddle equation gives
\begin{equation}\label{eq:B-volume-saddle-identities}
 \alpha(c_*)=
 \frac1{\lambda_*}-\frac{w_*}{e^{\lambda_*w_*}-1},\qquad
 Z_{\lambda_*}(w_*)=c_*e^{-c_*},\qquad
 \lambda_*\alpha(c_*)=\frac{1-e^{-c_*}}{c_*}.
\end{equation}
Moreover,
\[
 F_{\lambda_*}'(c)
 =\lambda_*\left(
   \frac{1-e^{-c}}{e^{\lambda_*w(c)}-1}-e^{-c}\right),
\]
whose sign is the sign of $c-\lambda_*w(c)$.  On $(0,\infty)$ the ratio
$c/w(c)$ is strictly decreasing, because
\[
 \left(\frac c{w(c)}\right)'
 =\frac{(1+c)e^{-c}-1}{w(c)^2}<0.
\]
Consequently $F_{\lambda_*}$ has its unique global maximum on
$(0,\infty)$ at $c_*$.

For the upper bound, exponential Markov gives, for every $\lambda>0$,
\[
 W_N(w_k(c),\alpha_k(c))
 \le e^{\lambda N\alpha_k(c)}Z_\lambda(w_k(c))^N.
\]
On $0\le c\le k/2$ one has $w_k(c)\le w(c)$ and
$\alpha_k(c)\le\alpha(c)+1/N$.  Taking $\lambda=\lambda_*$ and using the
maximum on $(0,\infty)$ just proved (the integrand vanishes at $c=0$)
therefore yields
\[
 \int_0^{k/2}W_N(w_k(c),\alpha_k(c))\,\mathrm d c
 \le \frac{k}{2}e^{\lambda_*}
     \exp\bigl(NF_{\lambda_*}(c_*)\bigr).
\]

For the matching lower bound, fix $0<w_0<w_*$ and let
\[
 Z_0:=Z_{\lambda_*}(w_0),\qquad
 \mu_0:=\frac1{\lambda_*}
        -\frac{w_0}{e^{\lambda_*w_0}-1}.
\]
The mean $\mu_0$ of the density
$e^{-\lambda_*x}/Z_0$ on $[0,w_0]$ is strictly smaller than
$\alpha(c_*)$: the first identity in
\eqref{eq:B-volume-saddle-identities} evaluates the same formula at
$w_*$, and $w\mapsto1/\lambda_*-w/(e^{\lambda_*w}-1)$ strictly
increases, since $x/(e^x-1)$ strictly decreases (again
$e^x(1-x)-1<0$).  Choose
$\varepsilon>0$ and a fixed interval $J$ about $c_*$ so that, for all
large $k$ and all $c\in J$,
\[
 w_k(c)\ge w_0,\qquad \alpha_k(c)\ge\mu_0+\varepsilon;
\]
this is possible because $s_k(c)\to e^{-c}$ locally uniformly.  If
$X_1,\ldots,X_N$ are independent with the preceding truncated-exponential
density, Chebyshev's inequality shows that
\[
 \Pr\left(\left|N^{-1}\sum_iX_i-\mu_0\right|\le\varepsilon\right)
 \ge\frac12
\]
for all large $N$.  Comparing this tilted probability with Lebesgue volume
on the same event gives, uniformly for $c\in J$,
\[
 W_N(w_k(c),\alpha_k(c))
 \ge\frac12\left[Z_0e^{\lambda_*(\mu_0-\varepsilon)}\right]^N.
\]
Integrating over $J$, then letting $w_0\uparrow w_*$ and
$\varepsilon\downarrow0$, proves
\[
 \lim_{k\to\infty}\frac1k\log
 \int_0^{k/2}W_N(w_k(c),\alpha_k(c))\,\mathrm d c
 =F_{\lambda_*}(c_*).
\]
By \eqref{eq:B-volume-saddle-identities},
\[
 F_{\lambda_*}(c_*)
 =\log c_*-c_*+\frac{1-e^{-c_*}}{c_*}.
\]
Combining this with \eqref{eq:B-volume-scaled} and
\eqref{eq:B-volume-prefactor} gives the logarithmic limit in
\eqref{eq:B-volume-asymptotic}.  The root limit follows by exponentiation,
since $\mathfrak v_k>0$ by Proposition~\ref{prop:B-interior}.
\end{proof}

\section{The Lean statements}\label{app:lean-statement}

For reference --- and for readers who want to see precisely what has
been machine-checked without opening the repository --- this appendix
reproduces, verbatim, the Lean statements of the quasigroup gates and of
the geometric results.  Each file below contains only definitions and
statements; the proofs live in separate modules of the repository.

The quasigroup statement file defines quasigroup convolution, the two
ordinary quasigroup gates, the absorbing-zero variant, and the combined
addition/multiplication gate of Theorem~\ref{thm:quasigroup}.

% Inlined verbatim from Rronce/Bq/QuasigroupStatement.lean (synced 2026-07-28).
% If the repository file changes, regenerate this block.
\begin{lstlisting}
import Rronce.Bq.Statement

/-!
# The B_q quasigroup gate, all q ≥ 4: STATEMENT ONLY

This proof-free audit file freezes the quasigroup extension of the addition
gate.  The implementation lives in `Rronce/Bq/QuasigroupProof.lean`, and
`Rronce/Bq/QuasigroupGate.lean` exposes the public theorems.

The statement uses the standard translation characterization of a
quasigroup: fixing either argument of the binary operation gives a
permutation of the carrier.  No identity, associativity, commutativity, or
inverse operation is assumed.

There are two genuinely different results:

* ordinary quasigroup convolution, which is the extension of the paper's
  addition gate;
* convolution for an operation with an absorbing zero whose restriction to
  the nonzero elements is a quasigroup, which is the extension of the
  paper's field-multiplication gate.

The distinction matters: multiplication on a whole field is not a
quasigroup operation, because translation by zero is constant.
-/

open scoped BigOperators

namespace Rronce
namespace Bq

noncomputable section

/--
A binary operation is a quasigroup operation when every left and right
translation is bijective.

Equivalently, each equation `op x y = z` has a unique solution in either one
of `x` and `y` after the other is fixed.
-/
def IsQuasigroupOp {Q : Type} (op : Q → Q → Q) : Prop :=
  (∀ x, Function.Bijective (op x)) ∧
  (∀ y, Function.Bijective (fun x => op x y))

/--
An operation is an absorbing-zero quasigroup operation when zero absorbs on
both sides and multiplication by every nonzero element is cancellative on
both sides.

On a finite carrier, the two injectivity clauses say precisely that the
nonzero left and right translations are permutations.  Together with
absorption they also force a product of two nonzero elements to be nonzero,
so the complement of `zero` is a quasigroup.  This is the exact algebraic
content of field multiplication used by the proof.
-/
def IsZeroQuasigroupOp {Q : Type} (zero : Q) (op : Q → Q → Q) : Prop :=
  (∀ x, op zero x = zero) ∧
  (∀ x, op x zero = zero) ∧
  (∀ x, x ≠ zero → Function.Injective (op x)) ∧
  (∀ y, y ≠ zero → Function.Injective (fun x => op x y))

/--
Convolution with respect to an arbitrary binary operation: the law of
`op X Y` for independent `X ∼ μ` and `Y ∼ ν`.

The double-sum definition does not build any quasigroup law into the
definition itself; `IsQuasigroupOp op` is a hypothesis of the preservation
theorem.
-/
def quasigroupConv {Q : Type} [Fintype Q] (op : Q → Q → Q)
    (μ ν : Q → ℝ) : Q → ℝ :=
  by
    classical
    exact fun z => ∑ x, ∑ y, if op x y = z then μ x * ν y else 0

/--
**Quasigroup convolution theorem.**  For every finite quasigroup of order at
least four, convolution preserves `B_q`.
-/
def QuasigroupConvPreservesBq : Prop :=
  ∀ (Q : Type) [Fintype Q] (op : Q → Q → Q),
    IsQuasigroupOp op →
    4 ≤ Fintype.card Q →
    ∀ μ ν : Q → ℝ,
      IsProbDist μ → IsProbDist ν →
      MemBq μ → MemBq ν →
      MemBq (quasigroupConv op μ ν)

/--
**Two-quasigroup gate theorem.**  If the operations labelled addition and
multiplication are both quasigroup operations on the same finite carrier,
then both convolution gates preserve `B_q`.
-/
def TwoQuasigroupGatesPreserveBq : Prop :=
  ∀ (Q : Type) [Fintype Q] (add mul : Q → Q → Q),
    IsQuasigroupOp add →
    IsQuasigroupOp mul →
    4 ≤ Fintype.card Q →
    ∀ μ ν : Q → ℝ,
      IsProbDist μ → IsProbDist ν →
      MemBq μ → MemBq ν →
      MemBq (quasigroupConv add μ ν) ∧
        MemBq (quasigroupConv mul μ ν)

/--
**Absorbing-zero quasigroup convolution theorem.**  For every finite
absorbing-zero extension of a quasigroup with at least four elements,
convolution preserves `B_q`.
-/
def ZeroQuasigroupConvPreservesBq : Prop :=
  ∀ (Q : Type) [Fintype Q] (zero : Q) (mul : Q → Q → Q),
    IsZeroQuasigroupOp zero mul →
    4 ≤ Fintype.card Q →
    ∀ μ ν : Q → ℝ,
      IsProbDist μ → IsProbDist ν →
      MemBq μ → MemBq ν →
      MemBq (quasigroupConv mul μ ν)

/--
**Quasigroup addition/multiplication gate theorem.**  On a finite carrier of
order at least four, suppose addition is a quasigroup operation and
multiplication has an absorbing zero and is a quasigroup off zero.  Then both
convolution gates preserve `B_q`.  No compatibility law between the two
operations is required.
-/
def QuasigroupGatesPreserveBq : Prop :=
  ∀ (Q : Type) [Fintype Q] (zero : Q) (add mul : Q → Q → Q),
    IsQuasigroupOp add →
    IsZeroQuasigroupOp zero mul →
    4 ≤ Fintype.card Q →
    ∀ μ ν : Q → ℝ,
      IsProbDist μ → IsProbDist ν →
      MemBq μ → MemBq ν →
      MemBq (quasigroupConv add μ ν) ∧
        MemBq (quasigroupConv mul μ ν)

end

end Bq
end Rronce
\end{lstlisting}

The volume statement file expresses the quantitative nonvacuity bound
\eqref{eq:B-volume-lower} in the coordinate chart described in
Appendix~\ref{sec:lean}.

% Inlined verbatim from Rronce/Bq/VolumeStatement.lean (synced 2026-07-28).
% If the repository file changes, regenerate this block.
\begin{lstlisting}
import Rronce.Bq.Statement

/-!
# A positive normalized fraction of `B_q`: STATEMENT ONLY

This proof-free audit file freezes the quantitative nonvacuity theorem used in
the paper.  The probability simplex on `q = m + 2` atoms is represented in
its standard `q - 1 = m + 1` dimensional coordinate chart: the displayed
coordinates are `x : Fin (m + 1) → ℝ`, and the omitted atom is
`1 - ∑ i, x i`.  Thus the ratio below is genuine relative simplex volume,
not the zero ambient `q`-dimensional volume of the affine hyperplane.

All proofs live outside this file, in `Rronce/Bq/VolumeProof.lean`; the public
header is `Rronce/Bq/Volume.lean`.
-/

open scoped BigOperators
open MeasureTheory

namespace Rronce
namespace Bq

noncomputable section

/-- Restore the omitted atom in the standard coordinate chart of the
probability simplex on `m + 2` points. -/
def completeLaw (m : ℕ) (x : Fin (m + 1) → ℝ) : Fin (m + 2) → ℝ :=
  Fin.cases (1 - ∑ i, x i) x

/-- The standard coordinate simplex: precisely the vectors whose completion
is a probability distribution. -/
def coordinateSimplex (m : ℕ) : Set (Fin (m + 1) → ℝ) :=
  {x | IsProbDist (completeLaw m x)}

/-- The pullback of the probability-distribution part of `B_(m+2)` to the
standard coordinate chart. -/
def coordinateBq (m : ℕ) : Set (Fin (m + 1) → ℝ) :=
  {x | IsProbDist (completeLaw m x) ∧ MemBq (completeLaw m x)}

/-- Normalized relative volume of `B_(m+2)`, as the coordinate volume of the
body divided by that of the whole probability simplex. -/
def volumeFraction (m : ℕ) : ℝ :=
  volume.real (coordinateBq m) / volume.real (coordinateSimplex m)

/-- Half the exact slack at the uniform law, in the closed form displayed in
the paper. -/
def fractionRadius (q : ℕ) : ℝ :=
  ((q : ℝ) - 1) ^ (q - 1) / (2 * (q : ℝ) ^ q)

/-- Both chart sets are Borel measurable, so the quotient is ordinary
Lebesgue volume rather than merely an outer-measure value. -/
def VolumeFractionMeasurable : Prop :=
  ∀ m : ℕ,
    MeasurableSet (coordinateSimplex m) ∧ MeasurableSet (coordinateBq m)

/-- **Quantitative geometric nonvacuity.**  For `q = m + 2 ≥ 4`, the normalized
`(q-1)`-dimensional volume fraction of `B_q` is at least

`(((q-1)^(q-1)) / (2*q^q))^(q-1)`.
-/
def VolumeFractionLowerBound : Prop :=
  ∀ m : ℕ, 2 ≤ m →
    fractionRadius (m + 2) ^ (m + 1) ≤ volumeFraction m

/-- The displayed strict-positivity consequence, frozen separately for audit. -/
def VolumeFractionPositive : Prop :=
  ∀ m : ℕ, 2 ≤ m → 0 < volumeFraction m
end
end Bq
end Rronce
\end{lstlisting}

The exact-volume statement file expresses the integral formula and the
sharp asymptotics of Appendix~\ref{app:exact-volume}.

% Inlined verbatim from Rronce/Bq/ExactVolumeStatement.lean (synced 2026-07-28).
% If the repository file changes, regenerate this block.
\begin{lstlisting}
import Rronce.Bq.VolumeStatement

/-!
# Exact relative volume and its exponential rate: STATEMENT ONLY

This proof-free audit file freezes the exact one-dimensional formula for the
normalized volume of `B_q` and the sharp exponential asymptotic statement.
The Lean index is `m = q - 2`, so the simplex has dimension `m + 1` and the
finite positive-part sum has order `m`.

All implementations live in separate proof modules.  The public header is
`Rronce/Bq/ExactVolume.lean`.
-/

open scoped BigOperators Interval
open Filter MeasureTheory

namespace Rronce
namespace Bq

noncomputable section

/-- Positive part, written `[x]_+` in the paper. -/
def positivePart (x : ℝ) : ℝ := max x 0

/-- The affine expression occurring in the clipped-cube
inclusion--exclusion formula. -/
def volumeSliceTerm (m j : ℕ) (t : ℝ) : ℝ :=
  1 - (j + 2 : ℕ) * t - (m - j : ℕ) * psi (m + 2) t

/-- The elementary integrand in the exact normalized-volume formula. -/
def exactVolumeIntegrand (m : ℕ) (t : ℝ) : ℝ :=
  ∑ j ∈ Finset.range (m + 1),
    (-1 : ℝ) ^ j * Nat.choose m j * positivePart (volumeSliceTerm m j t) ^ m

/-- The exact one-dimensional expression for the normalized fraction of the
probability simplex occupied by `B_(m+2)`. -/
def exactVolumeIntegral (m : ℕ) : ℝ :=
  ((m + 2 : ℕ) : ℝ) * ((m + 1 : ℕ) : ℝ) ^ 2 *
    ∫ t in (0 : ℝ)..(1 / 2 : ℝ), exactVolumeIntegrand m t

/-- **Exact volume formula.**  For every `q = m + 2 ≥ 4`, the normalized
relative volume of `B_q` equals `exactVolumeIntegral m`. -/
def ExactVolumeFormula : Prop :=
  ∀ m : ℕ, 2 ≤ m → volumeFraction m = exactVolumeIntegral m

/-- The scalar saddle equation selecting the exponential rate. -/
def IsVolumeSaddle (c : ℝ) : Prop :=
  0 < c ∧
    (1 - Real.exp (-c)) / c = 1 - c / (Real.exp c - 1)

/-- The logarithmic exponential rate at the saddle. -/
def volumeLogRate (c : ℝ) : ℝ :=
  Real.log c - c - 1 + (1 - Real.exp (-c)) / c

/-- **Sharp exponential asymptotics.**  There is a unique positive saddle
`c_*`, and the normalized fractions satisfy

`log (volumeFraction m) / (m+2) → volumeLogRate c_*`.

Equivalently, their `(m+2)`-th roots tend to
`exp (volumeLogRate c_*) = 0.2183305369...`.
-/
def ExactVolumeExponentialAsymptotics : Prop :=
  ∃ c : ℝ,
    IsVolumeSaddle c ∧
    (∀ d : ℝ, IsVolumeSaddle d → d = c) ∧
    Tendsto
      (fun m : ℕ ↦ Real.log (volumeFraction m) / ((m + 2 : ℕ) : ℝ))
      atTop (nhds (volumeLogRate c))

/-- The equivalent root-exponential rate. -/
def volumeRootRate (c : ℝ) : ℝ :=
  Real.exp (volumeLogRate c)

/-- Root form of the sharp exponential asymptotics.  This is frozen
separately so that the phrase "exponential rate" has a literal theorem:
the `(m+2)`-th root of the normalized volume converges. -/
def ExactVolumeRootAsymptotics : Prop :=
  ∃ c : ℝ,
    IsVolumeSaddle c ∧
    (∀ d : ℝ, IsVolumeSaddle d → d = c) ∧
    Tendsto
      (fun m : ℕ ↦
        Real.rpow (volumeFraction m) (1 / ((m + 2 : ℕ) : ℝ)))
      atTop (nhds (volumeRootRate c))

end
end Bq
end Rronce
\end{lstlisting}

\end{document}